\documentclass[11pt]{article}

\usepackage{mathtools}
\usepackage{color}
\usepackage{microtype}
\usepackage[usenames,dvipsnames,svgnames,table]{xcolor}
\usepackage{amsfonts,amsmath, amssymb,amsthm,amscd,mathrsfs}
\usepackage[height=9in,width=6.5in]{geometry}
\usepackage{verbatim}
\usepackage{tikz}
\usetikzlibrary{decorations.markings}
\tikzset{->-/.style={decoration={
			markings,
			mark=at position .6 with {\arrow{>}}},postaction={decorate}}}
\usepackage{tkz-euclide}

\usepackage[margin=10pt,font=small,labelfont=bf]{caption}
\usepackage{latexsym}
\usepackage[pdfauthor={ },bookmarksnumbered,hyperfigures,colorlinks=true,citecolor=BrickRed,linkcolor=BrickRed,urlcolor=BrickRed,pdfstartview=FitH]{hyperref}

\usepackage{graphicx}
\usepackage{enumitem}

\usepackage[inline,nolabel]{showlabels}

\usepackage[capitalize]{cleveref}
\Crefname{fact}{Fact}{Facts}
\crefname{theorem}{Theorem}{Theorems}
\crefname{thm}{Theorem}{Theorems}
\crefname{lemma}{Lemma}{Lemmas}
\crefname{claim}{Claim}{Claims}
\crefname{lem}{Lemma}{Lemmas}
\crefname{remark}{Remark}{Remarks}
\crefname{prop}{Proposition}{Propositions}
\crefname{defn}{Definition}{Definitions}
\crefname{corollary}{Corollary}{Corollaries}
\crefname{conjecture}{Conjecture}{Conjectures}
\crefname{question}{Question}{Questions}
\crefname{chapter}{Chapter}{Chapters}
\crefname{section}{Section}{Sections}
\crefname{part}{Part}{Parts}
\crefname{figure}{Figure}{Figures}

\newtheorem{theorem}{Theorem}[section]
\newtheorem{corollary}[theorem]{Corollary}
\newtheorem{lemma}[theorem]{Lemma}
\newtheorem{proposition}[theorem]{Proposition}
\newtheorem{question}[theorem]{Question}

\newtheorem{example}[theorem]{Example}

\newtheorem{definition}[theorem]{Definition}
\newtheorem{remark}[theorem]{Remark}

\numberwithin{equation}{section}

\newtheorem*{theorem*}{Theorem}
\newtheorem*{claim*}{Claim}
\newtheorem*{conj*}{Conjecture}
\newtheorem*{corollary*}{Corollary}
\newtheorem*{definition*}{Definition}
\newtheorem*{example*}{Example}
\newtheorem*{exercise*}{Exercise}
\newtheorem*{lemma*}{Lemma}
\newtheorem*{observation*}{Observation}
\newtheorem*{proposition*}{Proposition}
\newtheorem*{question*}{Question}
\newtheorem*{remark*}{Remark}

\newcommand{\bbE}{{\ensuremath{\mathbb E}} }

\newcommand{\bbN}{{\ensuremath{\mathbb N}} }

\newcommand{\bbP}{{\ensuremath{\mathbb P}} }

\newcommand{\bbR}{{\ensuremath{\mathbb R}} }

\newcommand{\bbZ}{{\ensuremath{\mathbb Z}} }

\newcommand{\bE}{{\ensuremath{\mathbf E}} }

\newcommand{\bP}{{\ensuremath{\mathbf P}} }

\newcommand{\cI}{{\ensuremath{\mathcal I}} }

\newcommand{\sM}{{\ensuremath{\mathscr M}} }

\newcommand{\sT}{{\ensuremath{\mathscr T}} }

\DeclareMathSymbol{\leqslant}{\mathalpha}{AMSa}{"36} 
\DeclareMathSymbol{\geqslant}{\mathalpha}{AMSa}{"3E} 
\DeclareMathSymbol{\eset}{\mathalpha}{AMSb}{"3F}     
\renewcommand{\le}{\;\leqslant\;}                   
\renewcommand{\ge}{\;\geqslant\;}                   
\renewcommand{\leq}{\;\leqslant\;}                   
\renewcommand{\geq}{\;\geqslant\;}                   

\newcommand{\be}{\begin{equation}}
	\newcommand{\ee}{\end{equation}}
\newcommand\ba{\begin{align}}
	\newcommand\ea{\end{align}}

\newcommand*\diff{\mathop{}\!\mathrm{d}}

\usepackage{nicematrix}
\usepackage{xcolor}  

\newcommand{\UF}{\mathbb{P}_{\mathrm{UF}}}

\newcommand{\UST}{\mathrm{UST}}
\newcommand{\PUST}{\bbP_\mathrm{UST}}
\newcommand{\EUST}{\bbE_\mathrm{UST}}

\newcommand{\MST}{\mathrm{MST}}
\newcommand{\PMST}{\bbP_\mathrm{MST}}
\newcommand{\EMST}{\bbE_\mathrm{MST}}

\newcommand{\PWIT}{\mathrm{U}}
\newcommand{\WMSF}{\mathrm{WMSF}}
\newcommand{\FMSF}{\mathrm{FMSF}}

\newcommand{\WP}{\mathrm{M}}

\DeclareMathOperator{\dist}{dist}

\newcommand{\Gb}{\mathcal{G}_\bullet}

\newcommand{\reff}{\mathcal{R}_{\mathrm{eff}}} 

\begin{document}
	
\author{
	Pengfei Tang\thanks{Center for Applied Mathematics and KL-AAGDM, Tianjin University, Tianjin, 300072, China.
		Email: \textsf{pengfei\_tang@tju.edu.cn}. Supported by the National Natural Science Foundation of China No. 12571151.}
	\qquad
	Zibo Zhang\thanks{School of Mathematics, Tianjin University, Tianjin, 300350, China.
		Email: \textsf{19932788533@163.com}.}
}
\date{\today}
\title{From second moments to pairwise negative correlation: Applications to minimal and uniform spanning trees}
\maketitle


\begin{abstract}
	We uncover a close connection between the second moment of the degree of a typical vertex in a random subgraph and the pairwise negative correlation (p-NC) property. On one hand, we exploit this connection to prove the p-NC property for non-adjacent edges in minimal spanning trees on complete graphs. On the other hand, we apply the classical p-NC property of uniform spanning trees to derive a universal upper bound on the second moment of the degree of a uniformly chosen vertex in uniform spanning trees on finite, connected, regular graphs, thereby resolving an open question posed by Nachmias and Peres. Furthermore, we determine that the optimal upper bound is exactly $6$, and the method for achieving this optimal bound is interesting in itself---the proof uses Edmonds' matroid polytope theorem.
\end{abstract}

%

\section{Introduction and main results}
Negative correlation is a fundamental probabilistic concept that captures the intuition that the occurrence of one event reduces the likelihood of another. Such negative dependence properties play a crucial role in the analysis of random combinatorial structures, with applications ranging from statistical mechanics to random graphs. 
Let $G=(V,E)$ be a finite graph (meaning  both $V$ and $E$ are finite sets).  We say a probability measure $\mu$ on $\Omega\coloneq \{0,1\}^E$ has the \emph{pairwise negative correlation} (p-NC) property if for any two distinct edges $e,f\in E$, 
\be\label{eq:def-p-NC}
\mu\big[  \omega(e)=\omega(f)=1 \big]\leq \mu\big[\omega(e)=1\big]\cdot \mu\big[\omega(f)=1\big]\,. 
\ee
We write $\eta(\omega)\coloneq \{ e\mid  \omega(e)=1 \}$ and often view $\omega$ as the subgraph $\big(V,\eta(\omega)\big)$. 
In \cite[Corollary~3.3]{TZ2026}, the authors observed two simple relations between the p-NC property of a random subgraph of a \textit{complete graph} and the second moment of the degree of a uniformly chosen vertex in that subgraph.  
In the present paper we generalize these relations to arbitrary finite, connected graphs and present two applications: one for the minimal spanning tree (MST) on complete graphs, and one for the uniform spanning tree (UST)  on finite, connected, regular graphs.

\subsection{Minimal spanning tree}
The first model we consider is the minimal spanning tree (MST)  on finite, connected graphs (a precise definition is given in Section~\ref{sec:2.1}). 
A natural question is whether the MST measure on $\{0,1\}^E$ satisfies the p-NC property. For uniform spanning trees, the p-NC property is well-established \cite[Section 4.2]{LP2016}. For MSTs, however, the situation is more nuanced: negative correlation can fail on some  graphs \cite[Section 5]{Lyons_Peres_Schramm2006minimal_spanning_forests}. For example, Lyons, Peres and Schramm \cite{Lyons_Peres_Schramm2006minimal_spanning_forests} constructed a modification of  $K_4$ (replacing each of two non-adjacent edges by three parallel edges) in which those two non-adjacent edges are positively correlated under the MST measure.

Nonetheless, since the MST measure is supported on spanning trees, the presence of an edge tends to create cycles which exclude other edges; this suggests that some negative dependence might persist in highly symmetric graphs. In particular, complete graphs $K_n$ are natural candidates for which one might expect the p-NC property to hold. In the present work we partially confirm this expectation by proving the p-NC property for pairs of non-adjacent edges in $K_n$ for all sufficiently large $n$.

\begin{theorem}\label{thm: NCMST}
There exists a constant $C>0$ such that for all $n\ge C$, the p-NC property holds for any pair of non-adjacent edges of $K_n$ under the MST measure $\PMST^{(K_n)}$. That is, if $e$ and $f$ are non-adjacent edges of $K_n$ with $n\ge C$, then 
\be\label{eq:NCMST}
\PMST^{(K_n)} \big[  \omega(e)=\omega(f)=1 \big]\leq \PMST^{(K_n)}\big[\omega(e)=1\big]\cdot \PMST^{(K_n)}\big[\omega(f)=1\big]\,.
\ee
\end{theorem}
\begin{remark}\label{rem:knn}
The same proof applies, with only minor modifications, to show that the MST measure on complete balanced bipartite graphs $K_{n,n}$ also satisfies the p-NC property for all pairs of non-adjacent edges whenever $n$ is sufficiently large.
\end{remark}

The first arXiv version of this paper established Theorem~\ref{thm: NCMST} and raised related questions concerning the remaining edge correlations. Subsequent work of Gupta \cite{Gupta2026} resolved these questions by establishing the full p-NC property for the MST on \(K_n\) for every \(n\ge 3\), together with a universal multiplicative correlation bound for MST measures on finite multigraphs; see Remark~\ref{rem:Gupta2026} for further discussion.

\subsection{Uniform spanning tree}
Given a finite, connected graph $G=(V,E)$, the uniform spanning tree measure $\PUST=\PUST(G)$ is the uniform probability measure on the set of spanning trees of $G$.
The $\UST$ is a classical and well-studied model in probability theory; its  p-NC property is standard (see e.g.~\cite[Chapter 4]{LP2016}).

In \cite[Section 7.2]{NP2022}, Nachmias and Peres investigated  moments of vertex degrees and sizes of graph-distance balls in  USTs on finite, connected, regular graphs. In particular, Lemma~7.3 in \cite{NP2022} states that there exists a constant $C>0$ such that for any finite, connected, regular graph $G$,  the degree of a uniformly random vertex $X$ in a UST sample $\omega$ on $G$ satisfies
\[
\PUST\big[\deg(X;\omega) \ge k \big]\leq \frac{C}{k^2} \quad \text{for all }\,\,k\geq 1\,.
\] 
This tail bound implies that the $p$-th moment of  $\deg(X;\omega)$ is bounded for all $p\in(1,2)$. They also provided an example showing that the $p$-th  moment need not be bounded when $p>2$. On page 1161 of \cite{NP2022}, Nachmias and Peres posed the open question  of whether the second moment  is uniformly bounded over all finite, connected, regular graphs. 
We  provide an affirmative answer to this question and give two proofs for this result. The first  proof is  concise  via an application of the p-NC property of USTs, while the second one is more complicated but yields the sharp bound $6$. 

\begin{theorem}\label{thm:UST}
There exists a constant $C>0$ such that for any finite,  connected, regular graph $G=(V,E)$, one has 
\[
\bbE\big[\deg(X;\omega)^2 \big]\leq C\,,
\]
where $\omega$ is a sample of  UST on $G$, the vertex $X$ is chosen uniformly at random from $V$ and independently of $\omega$, and $\deg(X;\omega)$ denotes the degree of $X$ in the UST sample $\omega$. 
\end{theorem}

\begin{theorem}\label{thm:UST-sharp}
The sharp constant in Theorem~\ref{thm:UST} is $6$. 
\end{theorem}

\subsection{Proof ideas and paper organization}

The three main results are driven by exact identities relating degree second moments to two-edge inclusion probabilities. These identities are established in Section~\ref{sec:relation} and used in two different directions: a lower bound on the degree second moment of the MST on $K_n$ yields p-NC for non-adjacent edges, while the p-NC property of the UST yields an upper bound on its degree second moment. The sharp UST bound requires a finer analysis of the same two-edge probabilities.

\noindent\textbf{The MST on the complete graph.}
Fix two non-adjacent edges of $K_n$, and let $p_0$ and $p_2$ denote the corresponding one-edge and two-edge inclusion probabilities as in Corollary~\ref{cor:1-2moment}. Edge transitivity gives $p_0=2/n$, while Corollary~\ref{cor: reduction to second moment bound} gives the exact equivalence
\[
p_2\leq p_0^2
\quad\Longleftrightarrow\quad
\EMST^{(K_n)}\big[\deg(x;\omega)^2\big]
\geq 5-\frac{11}{n}+\frac{6}{n^2}
\]
for every vertex $x$. It therefore suffices to prove that this second moment is strictly greater than $5$ for all sufficiently large $n$.

Let $N_n$ be the degree of a uniformly chosen vertex in $\MST(K_n)$, and let $N$ be the degree of the root in the root component of the WMSF on the PWIT. The local convergence result used in Proposition~\ref{prop: N_n to N}, together with Fatou's lemma, yields
\[
\liminf_{n\to\infty}\bbE[N_n^2]\geq \bbE[N^2].
\]
Proposition~\ref{prop:N^2value} computes the limiting moment exactly as
\[
\bbE[N^2]=10-4\zeta(3)\approx 5.191772>5.
\]
The computation in Proposition~\ref{prop:N^2value} uses the invasion percolation and Poisson Galton--Watson aggregation decomposition of the WMSF on the PWIT from \cite{Nachmias_Tang2024}. As pointed out by Luca Makowiec, the same identity also follows directly from Aldous's mixed-Poisson description of the limiting degree \cite[Proposition~2]{Aldous1990random_tree}. The strict inequality above proves Proposition~\ref{prop: lower bound on N_n^2}, and hence Theorem~\ref{thm: NCMST}. Gupta's subsequent finite-$n$ strengthening \cite{Gupta2026} is recorded separately in Remark~\ref{rem:Gupta2026} and is not used in this argument.

\noindent\textbf{A uniform second-moment bound for the UST.}
The cases $d\in\{1,2\}$ are immediate, so consider a finite connected simple $d$-regular graph with $d\geq3$. Since $X$ is uniform on $V$, it suffices to bound the sum of the expected squared degrees by a constant multiple of $|V|$. Corollary~\ref{cor:UST}, applied using the p-NC property of the UST, reduces this task to a sum of products of one-edge inclusion probabilities. Kirchhoff's formula and the Cauchy--Schwarz inequality then reduce it further to $\sum_{e\in E}\reff(e)^2$. The resistance-level decomposition in Corollary~\ref{cor:E_i}, obtained from \cite[Lemma~3.1]{NP2022}, yields
\[
\sum_{e\in E}\reff(e)^2\leq \frac{45|V|}{d}.
\]
For each fixed first edge there are $2(d-1)$ choices of an adjacent second edge, and the resulting ordered-pair sum yields $\sum_{x\in V}\EUST[\deg(x;\omega)^2]\leq92|V|$. Averaging over $X$ proves Theorem~\ref{thm:UST}.

\noindent\textbf{The sharp constant.}
The cases $d\in\{1,2\}$ are again immediate, so assume that $d\geq3$. For the upper bound, let $X$ be uniform on $V$ and let $Y,Z$ be two distinct uniformly chosen neighbors of $X$. Lemma~\ref{lem:appendix-second-moment} reduces the second moment to the probability that the two edges of this ordered wedge belong to the UST. If
\[
r_1=\reff(X\leftrightarrow Y;G),\qquad
r_2=\reff(X\leftrightarrow Z;G),\qquad
t=\reff(Y\leftrightarrow Z;G),
\]
then averaging the transfer-current formula of Lemma~\ref{lem:appendix-transfer-current} over the random wedge yields
\[
\bbE\left[r_1r_2-\frac14(r_1+r_2-t)^2\right].
\]
Write $a=2/(d+1)$ and $r_1=a+\alpha$, $r_2=a+\beta$, and $t=a+\gamma$. The universal effective-resistance lower bound makes the corrections $\alpha,\beta,\gamma$ nonnegative. Foster's theorem and Kirchhoff's node law determine the required linear moments, while Edmonds' matroid polytope theorem represents the vector of edge corrections as a convex combination of forest indicators. A deterministic forest degree estimate then controls the mixed term $\bbE[\alpha\beta]$. Discarding the remaining nonpositive square term and substituting these estimates gives
\[
\bbE\big[\deg(X;\omega)^2\big]
\leq 6-\frac{d+7}{(d+1)^2}-\frac{2}{d(d+1)}<6,
\]
which proves Proposition~\ref{prop:appendix-UST-sharp}.

To prove sharpness, for each odd $d$ we attach $d$ copies of a dense one-port graph $B_d$ to a new central vertex by bridge edges. The resulting graph $G_d$ is $d$-regular. A Moore--Penrose inverse calculation yields the uniform estimate
\[
\reff(u\leftrightarrow v;B_d)=\frac{2}{d+2}+O(d^{-2}),
\]
and hence an average squared UST degree of $5+O(d^{-1})$ inside each block. Every joining edge is present in every spanning tree, and the central vertex has degree $d$ deterministically; its contribution supplies the remaining asymptotic unit. Consequently,
\[
\lim_{d\to\infty}\bbE\big[\deg(X_d;\omega_d)^2\big]=6,
\]
proving Proposition~\ref{prop:appendix-tightness} and Theorem~\ref{thm:UST-sharp}.

\noindent\textbf{Organization.}
The remainder of the paper is organized as follows. Section~\ref{sec:relation} establishes general identities relating the second moment of vertex degrees to two-edge inclusion probabilities, and derives the corresponding implications for pairwise negative correlation. Section~\ref{sec:pf-MST} reviews the PWIT and WMSF framework, computes the limiting second moment of the root degree, and applies these results to prove Theorem~\ref{thm: NCMST}. Section~\ref{sec:UST} collects the electrical-network, Laplacian, and matroid tools needed for the UST setting, and then proves Theorem~\ref{thm:UST}, the strict upper bound in Proposition~\ref{prop:appendix-UST-sharp}, and the sharpness construction in Proposition~\ref{prop:appendix-tightness}.


\section{Relating the second moment to pairwise negative correlation}\label{sec:relation}

In this section, we establish general relations between the second moment of the degree and the p-NC property. The following lemma generalizes Lemma 3.2 in \cite{TZ2026} (restricted to the $k=1$ case) and its proof is similar. We include full details for the reader's convenience.

\begin{lemma}\label{lem:2nd-pNC}
Suppose $G=(V,E)$ is a finite, connected graph. Let $\sT(G)$ be the set of spanning trees of $G$ and let $\bP$ be an $\mathrm{Aut}(G)$-invariant probability measure supported on $\sT(G)$ ($\bE$ denotes the corresponding expectation); examples include the UST measure $\PUST$ and the MST measure $\PMST$. For $e,f\in E$, denote by $e\sim f$ if $e$ and $f$ are adjacent edges in $G$ (i.e., they share a common endpoint), and by $e\not\sim f$ if they are non-adjacent. Write $n=|V|$. For $x\in V$ and $\omega\in\{0,1\}^E$, let $\deg(x;\omega)$ denote the degree of $x$ in the subgraph $\big(V,\eta(\omega)\big)$. Then we have the equalities
\be\label{eq:relation-1-ad}
\sum_{\langle e,f\rangle \colon e\sim f}\bP\big[\omega(e)=\omega(f)=1\big]=\sum_{x\in V}\bE\big[ \deg(x;\omega)^2 \big]-2(n-1)\,,
\ee
where the sum on the left-hand side runs over all ordered pairs of distinct adjacent edges $e,f$; and
\be\label{eq:relation-2-nonad}
\sum_{\langle e,f\rangle \colon e\not\sim f}\bP\big[\omega(e)=\omega(f)=1\big]=n(n-1)-\sum_{x\in V}\bE\big[ \deg(x;\omega)^2 \big]\,,
\ee
where the sum on the left-hand side runs over all ordered non-adjacent pairs of edges.
\end{lemma}
\begin{proof}
First, since $\bP$ is supported on $\sT(G)$ and $|V|=n$, there are $n-1$ edges in $\eta(\omega)$ $\bP$-almost surely. Hence for $\bP$-almost every $\omega$,
\[
\sum_{x\in V}\deg(x;\omega)=2|\eta(\omega)|=2(n-1)\,.
\]
Second, for every configuration $\omega\in\Omega$,
\[
\sum_{\langle e,f\rangle \colon e\sim f} \mathbf{1}_{ \{ \omega(e)=\omega(f)=1  \} } =\sum_{x\in V}\deg(x;\omega)\big[ \deg(x;\omega)-1 \big]
=\sum_{x\in V}\deg(x;\omega)^2-\sum_{x\in V}\deg(x;\omega)\,.
\]
Combining the two observations, for $\bP$-almost every $\omega$,
\[
\sum_{\langle e,f\rangle \colon e\sim f} \mathbf{1}_{ \{ \omega(e)=\omega(f)=1  \} } =\sum_{x\in V}\deg(x;\omega)^2-2(n-1)\,.
\]
Taking expectations on both sides yields \eqref{eq:relation-1-ad}.

For \eqref{eq:relation-2-nonad}, observe that for $\bP$-almost every $\omega$ and an edge $e\in \eta(\omega)$, the number of edges $f\in\eta(\omega)\setminus\{e\}$ that are adjacent to $e=\{e^-,e^+\}$ is
\[
\deg(e^-;\omega)+\deg(e^+;\omega)-2\,,
\]
where $e^-$ and $e^+$ are the two endpoints of $e$. Hence the number of edges $f\in\eta(\omega)\setminus\{e\}$ that are not adjacent to $e$ equals
\[
|\eta(\omega)|-1-\big[\deg(e^-;\omega)+\deg(e^+;\omega)-2\big]=n-\deg(e^-;\omega)-\deg(e^+;\omega)\,.
\]
Therefore, for $\bP$-almost every $\omega$,
\begin{align*}
\sum_{\langle e,f\rangle \colon e\not\sim f} \mathbf{1}_{ \{ \omega(e)=\omega(f)=1  \} } &=\sum_{e\in\eta(\omega)}\big[n-\deg(e^-;\omega)-\deg(e^+;\omega)\big]\\
&=n\cdot |\eta(\omega)|-\sum_{e\in\eta(\omega)}\big[\deg(e^-;\omega)+\deg(e^+;\omega)\big]\\
&=n(n-1)-\sum_{x\in V}\deg(x;\omega)^2\,.
\end{align*}
Taking expectations gives \eqref{eq:relation-2-nonad}.
\end{proof}

\begin{corollary}\label{cor:UST}
With the notation of Lemma~\ref{lem:2nd-pNC},
if $\bP$ has the p-NC property for adjacent edges, then we obtain an upper bound on the second moment:
\[
\sum_{x\in V}\bE\big[ \deg(x;\omega)^2 \big]\leq 2(n-1)+\sum_{\langle e,f\rangle \colon e\sim f}\bP[\omega(e)=1]\cdot \bP[\omega(f)=1]\,.
\]
If $\bP$ has the p-NC property for non-adjacent edges, then we obtain a lower bound:
\[
\sum_{x\in V}\bE\big[ \deg(x;\omega)^2 \big]\geq  n(n-1)-\sum_{\langle e,f\rangle \colon e\not\sim f}\bP[\omega(e)=1]\cdot \bP[\omega(f)=1]\,.
\]
\end{corollary}

As noted, Corollary~\ref{cor:UST} will play a key role in the proof of Theorem~\ref{thm:UST}. We now present Corollary~\ref{cor:1-2moment}, which is a specialization of Lemma~\ref{lem:2nd-pNC} to the complete graph $K_n$ under the MST or UST measure. In this setting, a simple first-moment calculation yields $p_0 = \frac{2}{n}$. Corollary~\ref{cor: reduction to second moment bound} then translates the desired p-NC inequality for non-adjacent edges into a lower bound on the second moment of the degree. Specifically, to prove Theorem~\ref{thm: NCMST} it suffices to show that for sufficiently large $n$,
\[
\mathbb{E}_{\mathrm{MST}}^{(K_n)}\big[\deg(x;\omega)^2\big] > 5.
\]
This lower bound will be established in Proposition~\ref{prop: lower bound on N_n^2} via a local limit approach.

\begin{corollary}\label{cor:1-2moment}
Let $\bbP$ be either $\PMST$ or $\PUST$ on $K_n$, and let $\bbE$ be the corresponding expectation. Let $e$ and $f$ be adjacent edges, and let $e$ and $e'$ be non-adjacent edges. Set
\[
p_0 = \bbP[\omega(e)=1],\quad
p_1 = \bbP[\omega(e)=\omega(f)=1],\quad
p_2 = \bbP[\omega(e)=\omega(e')=1].
\]
Then
\[
p_0 = \frac{2}{n},\qquad
p_1 = \frac{\bbE\big[\sum_{x\in V(K_n)}\deg(x;\omega)^2\big]-2(n-1)}{n(n-1)(n-2)},
\]
and
\[
p_2 = \frac{4\big[n(n-1)-\bbE\big(\sum_{x\in V(K_n)}\deg(x;\omega)^2\big)\big]}{n(n-1)(n-2)(n-3)}.
\]
\end{corollary}

\begin{corollary}\label{cor: reduction to second moment bound}
With the notation of Corollary~\ref{cor:1-2moment}, let $x$ be an arbitrary vertex in $K_n$. Then 
\[
p_1 \le p_0^2 \quad\Longleftrightarrow\quad
\bbE\big[\deg(x;\omega)^2\big] \le 6-\frac{14}{n}+\frac{8}{n^2},
\]
and
\[
p_2 \le p_0^2 \quad\Longleftrightarrow\quad
\bbE\big[\deg(x;\omega)^2\big] \ge 5-\frac{11}{n}+\frac{6}{n^2}.
\]
Moreover, $p_2 = p_0^2$ if and only if $\bbE[\deg(x;\omega)^2] = 5-\frac{11}{n}+\frac{6}{n^2}$.
\end{corollary}

\begin{remark*}	
Analogues of Corollaries~\ref{cor:1-2moment} and \ref{cor: reduction to second moment bound}  remain valid for complete bipartite graphs \(K_{n,n}\). In particular, the p-NC property for non-adjacent edges in \(K_{n,n}\) is equivalent to an analogous lower bound on the second moment of vertex degrees: 
\[
\bbE[\deg(x;\omega)^2] \geq  5-\frac{13}{2n}+\frac{3}{n^2}-\frac{1}{2n^3}\,,
\]
which enables the large-$n$ result in Remark~\ref{rem:knn} via the same local limit approach.
\end{remark*}

\section{The p-NC property for MST on complete graphs}\label{sec:pf-MST}
This section establishes the p-NC property for non-adjacent edges in the minimal spanning tree (MST) on large complete graphs. We first review key definitions and background, then compute the second moment of the root degree in the wired minimal spanning forest (WMSF) on the Poisson weighted infinite tree (PWIT), and finally complete the proof of Theorem~\ref{thm: NCMST}.

\subsection{Preliminaries}\label{sec:MST_background}

\subsubsection{Minimal spanning trees and wired minimal spanning forests}\label{sec:2.1}
We briefly review the definitions and basic facts about MSTs and wired (and free) minimal spanning forests; further background can be found in \cite{Lyons_Peres_Schramm2006minimal_spanning_forests} or Chapter 11 of \cite{LP2016}. 

\begin{definition}\label{def:MST}
The \emph{minimal spanning tree} $\MST$ on a finite, connected graph $G=(V,E)$ is defined as follows. Sample i.i.d.~random variables $\{W(e)\colon e\in E \}$, each uniformly distributed on $(0,1)$. Let $\sT(G)$ be the set of spanning trees of $G$; the minimal spanning tree $\MST(G)$ is the unique spanning tree that minimizes the total weight (sum of edge weights) among all trees in $\sT(G)$. Denote by $\PMST=\PMST^{(G)}$ and $\EMST=\EMST^{(G)}$ the law and expectation of $\MST(G)$. 
\end{definition}

A classical algorithm for sampling MSTs is Prim's algorithm \cite{Prim1957shortest}. It relies on the following facts: for a finite, connected $G$ with a weight function $W:E\to(0,\infty)$ such that all edge weights are distinct, the associated $\MST(G;W)$ (the unique tree minimizing total weight) satisfies that an edge $e$ is absent from $\MST(G;W)$ if and only if it has the maximal weight on some cycle. Prim's algorithm starts with an arbitrary vertex $u$, sets $T_1=\{u\}$, and then repeatedly adds the lowest-weight edge connecting the current tree to its complement, producing $\MST(G;W)$.

If $G=(V,E)$ is a locally finite, infinite, connected graph with an injective weight function $W:E\to(0,\infty)$, we obtain the \emph{free minimal spanning forest} $\FMSF(G;W)$ by deleting every edge that is maximal on some cycle. When the weights are i.i.d. uniform $(0,1)$, the resulting random forest is called the \textit{free minimal spanning forest} (FMSF) of $G$; it arises as the weak limit of MSTs on an exhaustion of $G$.

For such an infinite graph, an \textit{extended cycle} is either a usual cycle or the union of two disjoint infinite simple paths emanating from a single vertex. Deleting every edge that is maximal on some extended cycle yields a spanning forest $\WMSF(G;W)$. With i.i.d. uniform $(0,1)$ weights, the resulting random forest is the \textit{wired minimal spanning forest} (WMSF) of $G$. The WMSF has the same law as the weak limit of MSTs on an exhaustion of $G$ with \textit{wired boundary conditions} (identifying all vertices outside the exhaustion and deleting loops while keeping parallel edges).

\subsubsection{Local limit}
The concept of local limit of finite graphs was introduced by Benjamini and Schramm \cite{Benjamini_Schramm2001local_limit}. We give only the necessary definitions; see \cite{AL2007} for further background.

\begin{definition}\label{def:rooted-graph}
A \textbf{rooted} graph $(G,\rho)$ consists of a locally finite graph $G=(V,E)$ and a distinguished vertex $\rho\in V$. Two rooted graphs are isomorphic if there is a graph isomorphism that maps one root to the other. Let $\Gb$ be the space of rooted graphs up to isomorphism; we always work with equivalence classes in $\Gb$.
\end{definition}

\begin{definition}\label{def:metric-rooted-graph}
Define $\dist\colon \Gb\times \Gb\to[0,1]$ by 
\[
\dist\big[(G_1,\rho_1),(G_2,\rho_2)\big]\coloneq  \frac{1}{1+\alpha}\,,
\]
where $\alpha$ is the supremum of radii $r>0$ such that the rooted balls $(B_{G_1}(\rho_1,r),\rho_1)$ and $(B_{G_2}(\rho_2,r),\rho_2)$ are isomorphic. Here $B_G(\rho,r)$ is the subgraph induced by vertices at graph distance $\le r$ from $\rho$. This metric makes $(\Gb,\dist)$ a Polish space \cite{AL2007}.
\end{definition}

\begin{definition}\label{def:local-limit}
Let $(G_n)_{n\ge1}$ be a sequence of locally finite, connected graphs (possibly random). Choose $\rho_n$ uniformly from $V(G_n)$, independently of each other. We say a random element $(G,\rho)\in \Gb$ is the \textbf{local limit} of $(G_n)_{n\ge1}$ if for every integer $r>0$, the rooted balls $(B_{G_n}(\rho_n,r),\rho_n)$ converge in distribution to $(B_G(\rho,r),\rho)$. 
\end{definition}

\subsubsection{The WMSF on the PWIT}
We briefly review the WMSF on the Poisson weighted infinite tree (PWIT); see \cite{Addario2013local_limit_MST_complete_graphs,Nachmias_Tang2024} for details.

The PWIT is a rooted weighted tree $(\PWIT,\emptyset,W)$ built on the Ulam-Harris tree $\PWIT$ (the canonical infinite rooted tree) with root $\emptyset$ (the empty sequence). Its vertex set is $V(\PWIT)=\bigcup_{k\ge0}\bbN^k$ (with $\bbN^0=\emptyset$), and each vertex $v=(n_1,\ldots,n_k)\in\bbN^k$ ($k\ge1$) is connected to its parent $(n_1,\ldots,n_{k-1})$. For every vertex $v$, the edges to its children are assigned weights $\{w_i\}_{i\ge1}$ that are the atoms of an independent homogeneous rate-$1$ Poisson point process on $[0,\infty)$. The processes for different vertices are independent; let $W=\big\{W(e)\colon e\in E(\PWIT) \big\}$ denote the collection of all edge weights.

The WMSF on the PWIT is defined by the edge-deletion rule for WMSFs: an edge is deleted if and only if its weight is maximal on some extended cycle. An equivalent constructive description uses invasion percolation: for each $u\in V(\PWIT)$, let $T^{(u)}$ be the invasion percolation cluster of $u$ (starting with $\{u\}$ and repeatedly adding the lowest-weight incident edge not yet in the cluster; this mimics Prim's algorithm in the infinite setting). The WMSF on the PWIT has the same law as the union of invasion percolation clusters $\bigcup_{u\in V(\PWIT)} T^{(u)}$ \cite[Proposition 3.3]{Lyons_Peres_Schramm2006minimal_spanning_forests}. 

Denote by $\WP$ the connected component of the root $\emptyset$ in this WMSF. The rooted version $(\WP,\emptyset)$ is the local limit of MSTs on complete graphs, and more generally on any sequence of finite regular graphs whose degrees tend to infinity \cite[Theorem 5.4]{Aldous_Steele2004objective_method} or \cite[Theorem 6.1]{Nachmias_Tang2024}. We now recall several facts from Section 2 of \cite{Nachmias_Tang2024}, which were adapted from \cite{Addario-Berry_etal2012invasion_PWIT} and \cite{Addario2013local_limit_MST_complete_graphs}. Almost surely, the component $M$ is a one-ended infinite tree and is the union of the invasion tree $T^{(\emptyset)}$ together with finite trees $M_v$ attached to its vertices. These finite trees $M_v$ can be sampled from a so-called Poisson Galton--Watson aggregation process (Definition~2.4 in \cite{Nachmias_Tang2024}). On the other hand, the invasion tree $T^{(\emptyset)}$ itself is a one-ended infinite tree and can be decomposed into finite ponds connected by outlets. Given $T^{(\emptyset)}$ together with the weights on its edges, the decomposition proceeds as follows:
\begin{itemize}
\item Let $e_1=(R_1,S_1)$ be the edge with the maximal weight in $T^{(\emptyset)}$, where $R_1$ is the endpoint of $e_1$ that is closer to $\emptyset$ in $T^{(\emptyset)}$. Let $P_1$ be the component of $\emptyset$ in the forest obtained from $T^{(\emptyset)}$ by deleting $e_1$. The edge $e_1$ is called the \emph{first outlet} and the finite tree $P_1$ the \emph{first pond}.
\item Define the remaining outlets and ponds recursively. Given $e_i=(R_i,S_i)$ for some $i\ge1$, let $e_{i+1}=(R_{i+1},S_{i+1})$ be the edge with the maximal weight on the infinite component of the forest obtained from $T^{(\emptyset)}$ by deleting $e_i$, and let $R_{i+1}$ be the endpoint closer to $\emptyset$. Let $P_{i+1}$ be the finite component containing $R_{i+1}$ in the forest obtained from $T^{(\emptyset)}$ by deleting $e_i$ and $e_{i+1}$. Then $e_{i+1}$ is the $(i+1)$-th outlet and $P_{i+1}$ the $(i+1)$-th pond.
\end{itemize}

For $\lambda>0$, let $\mathrm{PGW}(\lambda)$ denote a Galton-Watson tree with Poisson$(\lambda)$ offspring distribution (starting from one particle), and let $\theta(\lambda)\coloneq\bbP[|\mathrm{PGW}(\lambda)|=\infty]$ be its survival probability. Then (see e.g., Lemma~2.8 in \cite{Nachmias_Tang2024}) $\theta(\lambda)>0$ iff $\lambda>1$, and for $\lambda>1$,
\be\label{eq:theta-lambda-id}
1-\theta(\lambda)=e^{-\lambda\theta(\lambda)}.
\ee

Let $N$ be the degree of $\emptyset$ in $\WP$. Write $D_1$ for the degree of $\emptyset$ inside the first pond of $T^{(\emptyset)}$, $D_2$ for its degree in $M_\emptyset$ (the finite trees attached directly to $\emptyset$ outside the pond), and $D_3=\mathbf{1}_{\emptyset=R_1}$, where $R_1$ is the endpoint of the first outlet closer to $\emptyset$. Then
\be\label{eq:ND}
N = D_1 + D_2 + D_3.
\ee
Let $X_1$ be the weight of the first outlet $e_1=(R_1,S_1)$ and $Z_1$ the number of vertices in the first pond $P_1$. The following proposition collects some facts from Section 2 of \cite{Nachmias_Tang2024} (see also \cite{Addario-Berry_etal2012invasion_PWIT,Addario2013local_limit_MST_complete_graphs}).

\begin{proposition}\label{prop:deg-root-M}
\begin{enumerate}
	\item The weight $X_1$ has the distribution of $\theta^{-1}(F)$, where $F\sim\mathrm{Uniform}(0,1)$.
	\item Conditioned on $X_1$, the random variables $D_1$, $D_2$ and $D_3$ are independent.
	\item Conditioned on $X_1=\lambda>1$, the random variable $Z_1$ has distribution
	\[
	\bbP[Z_1=m\mid X_1=\lambda] = \frac{\theta(\lambda)}{\theta'(\lambda)}\cdot \frac{e^{-\lambda m}(\lambda m)^{m-1}}{(m-1)!},\qquad m=1,2,\dots
	\]
	\item Conditioned on $X_1$ and $Z_1$, the first pond is uniformly distributed among all labelled trees with $Z_1$ vertices, and the root $\emptyset$ is uniformly distributed among its vertices.
	\item Conditioned on $X_1=\lambda$ and $Z_1$, the random variable $D_2$ is Poisson with mean $\alpha(\lambda)\coloneq  \int_{\lambda}^{\infty}(1-\theta(x))\,dx$. Hence
	\[
	\bbE[D_2\mid X_1,Z_1] = \alpha(X_1),\qquad \bbE[D_2^2\mid X_1,Z_1] = \alpha(X_1)^2+\alpha(X_1).
	\]
	\item Conditioned on $X_1$ and $Z_1$, the random variable $D_3$ is Bernoulli with
	\[
	\bbE[D_3\mid X_1,Z_1] = \bbE[D_3^2\mid X_1,Z_1] = \frac{1}{Z_1}.
	\]
\end{enumerate}
\end{proposition}

\subsection{Proof of Theorem~\ref{thm: NCMST}}

As mentioned earlier, by Corollary~\ref{cor: reduction to second moment bound}, to prove Theorem~\ref{thm: NCMST} it suffices to establish the following proposition.
\begin{proposition}\label{prop: lower bound on N_n^2}
There exists a constant $C>0$ such that for all $n\ge C$,
\[
\EMST^{(K_n)}\big[\deg(x;\omega)^2\big] > 5,
\]
where $x$ is any vertex of $K_n$.
\end{proposition}

A key ingredient for Proposition~\ref{prop: lower bound on N_n^2} is the following local limit result.
\begin{proposition}\label{prop: N_n to N}
Let $\omega_n = \MST(K_n)$ be independent samples, let $\rho_n$ be a vertex chosen uniformly from $V(K_n)$ (independent of the $\omega_n$), and set $N_n = \deg(\rho_n;\omega_n)$. Let $N$ be the degree of the root $\emptyset$ in the WMSF on the PWIT. Then
\[
\liminf_{n\to\infty} \EMST^{(K_n)}\big[\deg(x;\omega)^2\big] = \liminf_{n\to\infty} \bE_n(N_n^2) \ge \bbE(N^2)\,,
\]
where $\bE_n$ denotes expectation with respect to the joint law of $\omega_n$ and $\rho_n$.
\end{proposition}
\begin{proof}
By symmetry, $\EMST^{(K_n)}[\deg(x;\omega)^2] = \bE(N_n^2)$ for any fixed $x$. It is known that $N_n \stackrel{d}{\to} N$ \cite[Theorem 5.4]{Aldous_Steele2004objective_method}. By Skorokhod's embedding theorem \cite{Skorokhod1956limit} we may assume $N_n \to N$ almost surely on a common probability space, and denote the corresponding expectation by $\bE$. Fatou's lemma then gives $\liminf \bE_n(N_n^2) =\liminf_{n\to\infty}\bE(N_n^2)\ge \bE(\liminf N_n^2) = \bbE(N^2)$. \end{proof}

\begin{question}\label{ques: N_n to N}
Is it true that $\lim_{n\to\infty} \EMST^{(K_n)}\big[\deg(x;\omega)^2\big] = \bbE(N^2)$?
\end{question}
A positive answer would also allow us to treat adjacent edges; see the discussion in the  remarks of Section~\ref{sec:MST_conclusion}. 

\begin{remark*}
Question~\ref{ques: N_n to N}, which was open in the first version of this paper, has since been answered affirmatively. Gupta~\cite{Gupta2026} obtained the convergence through a finite-$n$ identity relating the degree second moment to the expected total weight of the MST, while Makowiec~\cite{Makowiec2026} provided an alternative approach via asymptotics for fixed-forest inclusion probabilities. We return to these subsequent developments in Remark~\ref{rem:Gupta2026}.
\end{remark*}

\subsubsection{The second moment of the degree of the root in the WMSF on the PWIT}
\label{sec:MST_second_moment}
We now compute $\bbE(N^2)$. First, a known fact about uniform spanning trees.

\begin{lemma}\label{lem: second moment of degree in UST}
For a fixed vertex $u$ in $K_n$,
\[
\EUST\big[\deg(u;\UST(K_n))^2\big] = 5 - \frac{11}{n} + \frac{6}{n^2}.
\]
\end{lemma}
\begin{proof}
For non-adjacent edges $e,f$ in $K_n$, it is well known that $\PUST(e,f\in\UST(K_n)) = \PUST(e\in\UST(K_n))\PUST(f\in\UST(K_n))$ (e.g., see \cite[Theorem~1.2]{Stark2011}). Hence $p_2 = p_0^2$ in Corollary~\ref{cor:1-2moment} for the UST measure, and Corollary~\ref{cor: reduction to second moment bound} yields the formula. \end{proof}

\begin{lemma}\label{lem:N^2}
Let 
\[
\beta(\lambda)\coloneq \bbE[1/|\mathrm{PGW}(\lambda)|] = \sum_{m=1}^\infty \frac{1}{m}\bbP[|\mathrm{PGW}(\lambda)|=m]
\]
and recall $\theta(\lambda)=\bbP[|\mathrm{PGW}(\lambda)|=\infty]$. Then
\[
\bbE[N^2] = 5 - I_1 + 2I_2,
\]
where
\[
I_1 = \int_1^\infty \theta(\lambda)[1-\theta(\lambda)]\,d\lambda,\qquad
I_2 = \int_1^\infty \theta(\lambda)\beta(\lambda)\,d\lambda.
\]
\end{lemma}
\begin{proof}
From Proposition~\ref{prop:deg-root-M} and Lemma~\ref{lem: second moment of degree in UST}, conditioning on $X_1$ and $Z_1$ gives
\begin{align*}
\bbE[D_1\mid X_1,Z_1] &= 2\Bigl(1-\frac{1}{Z_1}\Bigr),\\
\bbE[D_1^2\mid X_1,Z_1] &= 5 - \frac{11}{Z_1} + \frac{6}{Z_1^2}.
\end{align*}
Using the conditional independence of $D_1,D_2,D_3$ given $X_1,Z_1$ and the formulas for $D_2$ and $D_3$ from Proposition~\ref{prop:deg-root-M}, we obtain
\begin{align*}
\bbE[N^2\mid X_1,Z_1] &=
\bbE[(D_1+D_2+D_3)^2\mid X_1,Z_1] \\
&= \bbE[D_1^2\mid X_1,Z_1] + \bbE[D_2^2\mid X_1,Z_1] + \bbE[D_3^2\mid X_1,Z_1] \\
&\quad + 2\bbE[D_1\mid X_1,Z_1]\bbE[D_2\mid X_1,Z_1] + 2\bbE[D_1\mid X_1,Z_1]\bbE[D_3\mid X_1,Z_1] \\
&\quad + 2\bbE[D_2\mid X_1,Z_1]\bbE[D_3\mid X_1,Z_1] \\
&= \Bigl(5-\frac{11}{Z_1}+\frac{6}{Z_1^2}\Bigr) + \bigl(\alpha(X_1)^2+\alpha(X_1)\bigr) + \frac{1}{Z_1} \\
&\quad + 2\cdot 2\Bigl(1-\frac{1}{Z_1}\Bigr)\alpha(X_1) + 2\cdot 2\Bigl(1-\frac{1}{Z_1}\Bigr)\frac{1}{Z_1} + 2\alpha(X_1)\frac{1}{Z_1} \\
&= 5 - \frac{6}{Z_1} + \frac{2}{Z_1^2} + \alpha(X_1)^2 + 5\alpha(X_1) - 2\frac{\alpha(X_1)}{Z_1}.
\end{align*}

Now compute the conditional expectations of $1/Z_1$ and $1/Z_1^2$ given $X_1=\lambda>1$. First, we have (see p.~951 of \cite{Addario-Berry_etal2012invasion_PWIT}) 
\be\label{eq:PGW-lambda}
\bbP[|\mathrm{PGW}(\lambda)|=m] = \frac{e^{-\lambda m}(\lambda m)^{m-1}}{m!},\qquad m\ge1,
\ee
and therefore
\be\label{eq:theta-expansion}
\theta(\lambda) = 1 - \sum_{m=1}^\infty \frac{e^{-\lambda m}(\lambda m)^{m-1}}{m!}.
\ee
Using item~3 of Proposition~\ref{prop:deg-root-M}, we obtain
\begin{align*}
\bbE\Bigl[\frac{1}{Z_1}\Bigm| X_1=\lambda\Bigr]
&=\sum_{m=1}^{\infty}\frac{1}{m}\cdot \frac{\theta(\lambda)}{\theta'(\lambda)}\cdot \frac{e^{-\lambda m}(\lambda m)^{m-1}}{(m-1)!}\\
&=\frac{\theta(\lambda)}{\theta'(\lambda)}\cdot\sum_{m=1}^{\infty}  \frac{e^{-\lambda m}(\lambda m)^{m-1}}{m!}\stackrel{\eqref{eq:theta-expansion}}{=}\frac{\theta(\lambda)}{\theta'(\lambda)}\cdot\big[1-\theta(\lambda)\big]\,,
\end{align*}
and
\begin{align*}
\bbE\big[\frac{1}{Z_1^2}\mid X_1=\lambda\big]
&=\sum_{m=1}^{\infty}\frac{1}{m^2}\cdot \frac{\theta(\lambda)}{\theta'(\lambda)}\cdot \frac{e^{-\lambda m}(\lambda m)^{m-1}}{(m-1)!}\\
&=\frac{\theta(\lambda)}{\theta'(\lambda)}\cdot\sum_{m=1}^{\infty}\frac{1}{m}\cdot \bbP\big[\big|\mathrm{PGW}(\lambda)\big|=m\big]=\frac{\theta(\lambda)\beta(\lambda)}{\theta'(\lambda)}\,,
\end{align*}
where in the last equality we used the definition $\beta(\lambda)\coloneq \bbE[1/|\mathrm{PGW}(\lambda)|]$.

The density of $X_1$ is $\theta'(\lambda)$ because $\bbP(X_1\le\lambda)=\theta(\lambda)$ for $\lambda>1$. Substituting into the expression for $\bbE[N^2\mid X_1,Z_1]$ and integrating yields
\[
\bbE[N^2] = 5 - 6I_1 + 2I_2 + I_3 + 5I_4 - 2I_5,
\]
where
\[
I_1 = \int_1^\infty \theta(\lambda)[1-\theta(\lambda)]\,d\lambda,\quad
I_2 = \int_1^\infty \theta(\lambda)\beta(\lambda)\,d\lambda,\quad
I_3 = \int_1^\infty \alpha(\lambda)^2\theta'(\lambda)\,d\lambda,
\]
\[
I_4 = \int_1^\infty \alpha(\lambda)\theta'(\lambda)\,d\lambda,\quad
I_5 = \int_1^\infty \theta(\lambda)[1-\theta(\lambda)]\alpha(\lambda)\,d\lambda.
\]
Using $\alpha(\infty)=0=\theta(1)$ and $\alpha'(\lambda)=-[1-\theta(\lambda)]$, integration by parts gives
\[
I_4 = \int_1^\infty \theta(\lambda)[1-\theta(\lambda)]\,d\lambda = I_1,\qquad
I_3 = -\int_1^\infty \theta(\lambda)\,2\alpha(\lambda)\alpha'(\lambda)\,d\lambda = 2I_5.
\]
Hence $\bbE[N^2] = 5 - I_1 + 2I_2$. \end{proof}

\begin{proposition}\label{prop:N^2value}
For the degree $N$ of the root $\emptyset$ in the WMSF on the PWIT, we have
\[
\bbE(N^2) =10-4\zeta(3)\approx {5.191772} \,,
\]
where $\zeta(3)\coloneq \sum_{k=1}^{\infty}\frac{1}{k^3}$ is the Ap\'{e}ry constant.
\end{proposition}

\begin{remark*}
The value $\bbE(N^2)$ is reminiscent of Frieze's famous $\zeta(3)$ theorem for minimal spanning trees \cite{Frieze1985value_MST}.
\end{remark*}

\begin{lemma}\label{lem:beta-q}
Define $q(\lambda)\coloneq 1-\theta(\lambda)$ for $\lambda>1$. Then for $\lambda>1$,
\be\label{eq:beta-q-relation}
\beta(\lambda)=q(\lambda)-\frac{\lambda q(\lambda)^2}{2}\,.
\ee
\end{lemma}
\begin{proof}
Fix $\lambda>1$ and set
\[
z\coloneq \lambda e^{-\lambda}\in(0,e^{-1})\,.
\]
Let $T(z)$ and $U(z)$ be the exponential generating functions for rooted trees and unrooted trees:
\[
T(z)\coloneq \sum_{m=1}^{\infty}\frac{m^{m-1}}{m!}z^m,\qquad 
U(z) \coloneq \sum_{m=1}^{\infty}\frac{m^{m-2}}{m!}z^m.
\]
First observe that
\begin{align*}
q(\lambda)&=1-\theta(\lambda)\stackrel{\eqref{eq:theta-expansion}}{=}\sum_{m=1}^{\infty}\frac{e^{-\lambda m} (\lambda m)^{m-1}}{m!}\\
&=\frac{1}{\lambda}\sum_{m=1}^{\infty}\frac{m^{m-1}}{m!}\big(\lambda e^{-\lambda}\big)^{m}=\frac{1}{\lambda}T(z)\,.
\end{align*}
Similarly,
\begin{align*}
\beta(\lambda)&=\sum_{m=1}^{\infty}\frac{1}{m}\bbP[|\mathrm{PGW}(\lambda)|=m] \stackrel{\eqref{eq:PGW-lambda}}{=}\sum_{m=1}^{\infty}\frac{1}{m}\cdot \frac{e^{-\lambda m} (\lambda m)^{m-1}}{m!}\\
&=\frac{1}{\lambda}\sum_{m=1}^{\infty}\frac{m^{m-2}}{m!}\big(\lambda e^{-\lambda}\big)^{m}=\frac{1}{\lambda}U(z)\,.
\end{align*}
Since $U(z)=T(z)-\frac{T^2(z)}{2}$ (see e.g., Theorem 1 of \cite{FSS2004}), we obtain \eqref{eq:beta-q-relation}:
\[
\beta(\lambda)=\frac{1}{\lambda}U(z)=\frac{1}{\lambda}\big[T(z)-\frac{T^2(z)}{2}\big]=q(\lambda)-\frac{\lambda q(\lambda)^2}{2}\,. \qedhere
\]
\end{proof}

\begin{proof}[Proof of Proposition~\ref{prop:N^2value}]

Using \eqref{eq:theta-lambda-id} and $q(\lambda)=1-\theta(\lambda)$, we have
\be\label{eq:q-l2}
\log q(\lambda)=-\lambda(1-q(\lambda))\,.
\ee
Differentiating with respect to $\lambda$ gives
\[
\frac{q'(\lambda)}{q(\lambda)}=-(1-q(\lambda))+\lambda q'(\lambda),
\]
hence
\[
q'(\lambda)=-\frac{q(\lambda)(1-q(\lambda))}{1-\lambda q(\lambda)}.
\]
Since $q(\lambda)\in(0,1)$  and $q'(\lambda)=-\theta'(\lambda)<0$ for $\lambda\in(1,\infty)$, we obtain 
\be\label{eq:q1}
0<q(\lambda)<\frac{1}{\lambda}\,,\,\,\forall\,\lambda\in(1,\infty)\,.
\ee
By Lemma~\ref{lem:N^2}, $\bbE(N^2)=5+2I_2-I_1$, where 
\[
2I_2-I_1=\int_{1}^{\infty}\theta(\lambda)\big[ 2\beta(\lambda)-q(\lambda)\big]\diff \lambda\,.
\]
From \eqref{eq:q1} and \eqref{eq:beta-q-relation}, for $\lambda>1$,
\[
2\beta(\lambda)-q(\lambda)=q(\lambda)-\lambda q(\lambda)^2=q(\lambda)\big[1-\lambda q(\lambda)\big]>0\,.
\]
Hence $	2I_2-I_1=\int_{1}^{\infty}\theta(\lambda)\big[ 2\beta(\lambda)-q(\lambda)\big]\diff \lambda>0$. While this suffices for the purpose of Proposition~\ref{prop: lower bound on N_n^2}, we continue to deduce the exact value of $2I_2-I_1$.
Since $q(\lambda)\downarrow 0$ as $\lambda\to\infty$ and $q(1)=1$, the change of variables $q=q(\lambda)$ yields
\begin{align*}
2I_2-I_1
&=
\int_1^\infty (1-q)q(1-\lambda q)\diff\lambda \\
&=
\int_1^\infty (1-q)q(1-\lambda q)
\left(-\frac{1-\lambda q}{q(1-q)}\right)\diff q \\
&=
\int_0^1 \bigl(1-\lambda(q)q\bigr)^2 \diff q,
\end{align*}
where
\[
\lambda(q)\stackrel{\eqref{eq:q-l2}}{=}\frac{-\log q}{1-q},\qquad q\in(0,1).
\]
Thus
\[
2I_2-I_1
=
\int_0^1 \left(1+\frac{q\log q}{1-q}\right)^2 \diff q.
\]
Substitute $q=e^{-t}$: $t\in(0,\infty)$, $\diff q=-e^{-t}\diff t$, and
\[
1+\frac{q\log q}{1-q}
=
1-\frac{te^{-t}}{1-e^{-t}}
=
1-\frac{t}{e^t-1}.
\]
Therefore
\[
2I_2-I_1
=
\int_0^\infty e^{-t}\left(1-\frac{t}{e^t-1}\right)^2\diff t.
\]
Expanding the square,
\[
2I_2-I_1
=
\int_0^\infty e^{-t}\diff t
-2\int_0^\infty \frac{te^{-t}}{e^t-1}\diff t
+\int_0^\infty \frac{t^2e^{-t}}{(e^t-1)^2}\diff t.
\]
For $t>0$, we have the absolutely convergent expansions
\[
\frac{e^{-t}}{e^t-1}
=
\frac{e^{-2t}}{1-e^{-t}}
=
\sum_{n=2}^\infty e^{-nt},
\]
and
\[
\frac{e^{-t}}{(e^t-1)^2}
=
\frac{e^{-3t}}{(1-e^{-t})^2}
=
\sum_{n=3}^\infty (n-2)e^{-nt}.
\]
By Tonelli's theorem,
\begin{align*}
2I_2-I_1
&=
1
-2\sum_{n=2}^\infty \int_0^\infty t e^{-nt}\diff t
+\sum_{n=3}^\infty (n-2)\int_0^\infty t^2 e^{-nt}\diff t \\
&=
1
-2\sum_{n=2}^\infty \frac{1}{n^2}
+2\sum_{n=3}^\infty \frac{n-2}{n^3}.
\end{align*}
Using
\[
\sum_{n=2}^\infty \frac{1}{n^2}=\zeta(2)-1
\]
and
\[
\sum_{n=3}^\infty \frac{n-2}{n^3}
=
\sum_{n=3}^\infty \frac{1}{n^2}
-2\sum_{n=3}^\infty \frac{1}{n^3}
=
\left(\zeta(2)-\frac{5}{4}\right)-2\left(\zeta(3)-\frac{9}{8}\right)
=
\zeta(2)+1-2\zeta(3),
\]
we conclude that
\begin{align*}
2I_2-I_1
&=
1-2(\zeta(2)-1)+2\bigl(\zeta(2)+1-2\zeta(3)\bigr) \\
&=
5-4\zeta(3).
\end{align*}
Using the numerical value $\zeta(3)\approx1.2020569$, we obtain  $\bbE(N^2)=10-4\zeta(3)\approx 5.191772$ .
This completes the proof.
\end{proof}

\subsubsection{Proof of Theorem~\ref{thm: NCMST} and supplementary remarks}\label{sec:MST_conclusion}

\begin{proof}[Proof of Proposition~\ref{prop: lower bound on N_n^2}]
By Proposition~\ref{prop:N^2value} we have $\bbE(N^2)>5$.  Combining this with  Proposition~\ref{prop: N_n to N} yields
\[
\liminf_{n\to\infty} \EMST^{(K_n)}\big[\deg(x;\omega)^2\big] \ge \bbE(N^2) > 5.
\]
Therefore there exists $C>0$ such that  for all  $n\geq C$, we have $\EMST^{(K_n)}[\deg(x;\omega)^2] > 5$, which proves the proposition. \end{proof}

\begin{proof}[Proof of Theorem~\ref{thm: NCMST}]
This follows immediately by combining Corollary~\ref{cor: reduction to second moment bound} with Proposition~\ref{prop: lower bound on N_n^2}. 
\end{proof}

\noindent\textbf{Remarks.}
%
Our proof of Theorem~\ref{thm: NCMST} uses only the lower bound on the degree second moment obtained from the local limit. Corollary~\ref{cor: reduction to second moment bound} also shows why the same argument does not directly yield the adjacent-edge inequality: for adjacent edges, one would instead need the finite-\(n\) upper bound
\[
\mathbb{E}\left[\deg(x;\omega)^2\right]
\le 6-\frac{14}{n}+\frac{8}{n^2}.
\]
Although Proposition~\ref{prop:N^2value} gives the limiting inequality
\[
\mathbb{E}[N^2]=10-4\zeta(3)<6,
\]
the one-sided estimate supplied by Fatou's lemma in Proposition~\ref{prop: N_n to N} is in the opposite direction and therefore does not provide the required finite-\(n\) upper bound. This explains the asymmetry between the adjacent and non-adjacent cases in our local-limit approach.

Subsequent work has resolved several questions that remained open in the first version of this paper. In particular, Gupta \cite{Gupta2026} proved the full p-NC property for the MST on \(K_n\) for every \(n\ge3\), thereby covering both adjacent and non-adjacent pairs. His finite-\(n\) identity also implies convergence of the degree second moment to the PWIT value computed in Proposition~\ref{prop:N^2value}; see Remark~\ref{rem:Gupta2026} for details.

We close this section with a question suggested to the first author by Russ Lyons, which asks whether MST measures satisfy a uniform multiplicative negative-dependence bound analogous to known inequalities for uniform forests. For the uniform forest measure 
$\UF$
(the uniform probability measure on all spanning forests of a finite, connected graph),  Br\"and\'en and Huh \cite{Branden_Huh2020} proved the inequality:
\[
\UF\big[  \omega(e)=\omega(f)=1 \big]\leq 2 \cdot \UF\big[\omega(e)=1\big]\cdot \UF\big[\omega(f)=1\big]\,,\quad \forall\, e\neq f\,.
\]
This motivated the following question of Russ Lyons, which appeared in the first version of this paper:
\begin{question}\label{ques:Russ}
Does there exist a constant $C>0$ such that for every finite, connected graph $G=(V,E)$ and every pair of distinct edges $e,f\in E$, the associated $\MST$ measure $\PMST$ satisfies 
\[
\PMST\big[  \omega(e)=\omega(f)=1 \big]\leq C \cdot \PMST\big[\omega(e)=1\big]\cdot \PMST\big[\omega(f)=1\big]\,?
\]
\end{question}

\begin{remark}\label{rem:Gupta2026}
Gupta~\cite{Gupta2026} subsequently answered Question~\ref{ques:Russ} affirmatively. More precisely, for every finite connected multigraph with i.i.d. edge weights having a common atomless distribution, he proved the universal bound
\[
\PMST\big[  \omega(e)=\omega(f)=1 \big]\leq 8 \,\PMST\big[\omega(e)=1\big]\cdot \PMST\big[\omega(f)=1\big]
\]
for all distinct edges \(e,f\). For complete graphs, he obtained the stronger result that the MST measure on \(K_n\) satisfies the full p-NC property for every \(n\ge3\), with strict inequality for adjacent pairs when \(n\ge3\) and for non-adjacent pairs when \(n\ge4\).  In addition, Gupta proved the finite identity
\[
\EMST^{(K_n)}\big[\deg(x;\omega)^2\big]
=\frac{10(n-1)}{n}-4\,\bE[L_n],
\]
where $L_n$ is the total weight of the MST of $K_n$ with i.i.d. rate-one exponential edge weights. Together with the known asymptotics for $\bE[L_n]$, this gives
\[
\lim_{n\to\infty}\EMST^{(K_n)}\big[\deg(x;\omega)^2\big]
=10-4\zeta(3)=\bbE(N^2),
\]
and hence gives an affirmative answer to Question~\ref{ques: N_n to N}. Makowiec \cite{Makowiec2026} independently obtained the same limiting second moment through asymptotics for fixed-forest inclusion probabilities.
\end{remark}


\section{The second moment of a typical vertex in UST on regular graphs}\label{sec:UST}

In this section, we use the p-NC property of uniform spanning trees (UST) to prove a universal upper bound on the second moment of the degree of a randomly chosen vertex in the UST of any finite, connected, regular graph. We further show this bound is exactly $6$, which is sharp. The proofs combine effective resistance estimates, the transfer-current theorem, and Edmonds’ matroid polytope theorem.

\subsection{Preliminaries}\label{sec:UST_prelim}
We collect key background on electrical networks, graph Laplacians, and matroid theory that are essential for the proofs in this section.

\subsubsection{Electrical networks and uniform spanning tree}\label{sec:elec}

Let $G=(V,E)$ be a finite, connected graph. We view $G$ as an electrical network where each edge has unit resistance. Define the Dirichlet form $\mathcal{E}$ by
\[
\mathcal{E}(f,g)\coloneq \sum_{e=\{e^-,e^+\}\in E}\big[f(e^+)-f(e^-)\big]\cdot \big[g(e^+)-g(e^-)\big]\,,\quad \forall\,f,g\in\bbR^{V}\,.
\]
For two distinct vertices $x,y\in V$, the \textbf{effective resistance} $\reff(x\leftrightarrow y;G)$ between $x$ and $y$ in $G$ is given by
\[
\reff(x\leftrightarrow y;G)^{-1} \coloneq \inf_{g\in \bbR^V}\{ \mathcal{E}(g,g)  \mid g(x)=1,g(y)=0   \}\,.
\]

Electrical networks are closely related to random walks and uniform spanning trees.
Kirchhoff's effective resistance formula \cite{K1847} states that for a finite, connected graph $G=(V,E)$, the probability that an edge $e$ belongs to the UST equals the effective resistance between its endpoints:
\be\label{eq:K-formula}
\PUST\big[\omega(e)=1\big]=\reff\big(e^{-}\leftrightarrow e^{+};G\big)\,.
\ee
Summing \eqref{eq:K-formula} over $e\in E$ yields Foster's theorem \cite{Fos1949}:
\be\label{eq:Foster}
\sum_{e\in E}\reff\big(e^{-}\leftrightarrow e^{+};G\big)=|V|-1\,.
\ee
For finite regular graphs, Nachmias and Peres used Foster's theorem to derive a useful quantitative estimate on the number of edges with relatively large effective resistance between their endpoints \cite[Lemma~3.1]{NP2022}.
\begin{lemma}[Lemma~3.1 in \cite{NP2022}]\label{lem:NP2022}
On any $d$-regular graph $G$ and for any $\varepsilon>\frac{2}{d}$, the number of edges $e=\{e^-,e^+\}$ with $\reff\big(e^-\leftrightarrow e^+\big)\geq \varepsilon$ is at most $\frac{|V(G)|}{\varepsilon d-2}$.
\end{lemma}
We will use the following corollary of Lemma~\ref{lem:NP2022}. For simplicity, write $\reff(e)$ for $\reff\big(e^{-}\leftrightarrow e^{+};G\big)$.
\begin{corollary}\label{cor:E_i}
Suppose $G=(V,E)$ is a finite, connected, regular graph with degree $d\ge 3$. Set $M\coloneq \lfloor \log_3d \rfloor$. For $i=1,\ldots,M$, set $\epsilon_i\coloneq \frac{3^i}{d}$ and set $\epsilon_{M+1}\coloneq 1$. Let $E_0$ be the set of edges $e$ such that $\reff(e)\leq \frac{3}{d}$. For $i=1,\ldots,M$, let $E_i$ be the set of edges $e$ with $\reff(e)\in (\epsilon_i,\epsilon_{i+1}]$. Then $E=\bigcup_{i=0}^{M}E_i$ and for each $i\in[1,M]$,
\[
|E_i|\leq \frac{|V|}{3^i-2}\leq \frac{|V|}{3^{i-1}}\,.
\] 
\end{corollary}

Orient each edge of $G$ in both directions and denote the set of oriented edges by $\overrightarrow{E}$. 
For  oriented edges $\vec{g},\vec{h}\in\overrightarrow{E}$, let $Y(\vec{g},\vec{h})$ be the  current  across $\vec{h}$ when a unit current is injected at the tail of $\vec{g}$ and extracted at the head of $\vec{g}$ (the current is measured in the direction of $\vec{h}$).  For a subset of edges $\{e_1,\ldots,e_k\}\subset E$, choose arbitrary fixed orientations $\{\vec{e}_1,\ldots,\vec{e}_k\}$ of these edges.   The transfer-current theorem of Burton and Pemantle \cite{BP1993} states that,
\be\label{eq:trans}
\mathbb{P}_{\mathrm{UST}}\big[e_i\in \eta(\omega)\,, \text{ for all }i=1,\ldots,k\big]
=
\det Y\,,
\ee
where $Y$ is a $k\times k$ matrix with entries $Y_{i,j}=Y(\vec{e}_i,\vec{e}_j)$.

The following lemma expresses the joint inclusion probability for two adjacent edges in terms of effective resistances; this is a standard consequence of the transfer-current theorem.
\begin{lemma}\label{lem:appendix-transfer-current}
Let $G=(V,E)$ be a finite, connected graph with unit resistances on all edges. Let $x,y,z\in V$
with $y$ and $z$ distinct neighbors of $x$. Let $e=\{x,y\}$ and $f=\{x,z\}$ be the two edges incident to $x$ with endpoints $y$ and $z$, respectively. 
Set
\[
r_1\coloneq \reff(x\leftrightarrow y;G),\qquad
r_2\coloneq \reff(x\leftrightarrow z;G),\qquad
t\coloneq \reff(y\leftrightarrow z;G).
\]
Let $\omega$ be a UST sample on $G$. Then 
\[
\mathbb{P}_{\mathrm{UST}}\big[e,f\in \eta(\omega)\big]
=
r_1r_2-\frac{1}{4}(r_1+r_2-t)^2.
\]
\end{lemma}

\begin{proof}
Let $\vec{e}$ and $\vec{f}$ be fixed orientations of $e$ and $f$ such that their tails are both $x$. 
Since each edge has unit resistance, 
we have
\[
Y(\vec{e},\vec{e}) =  \reff(x\leftrightarrow y;G) = r_1,
\]
and
\[
Y(\vec{f},\vec{f}) = \reff(x\leftrightarrow z;G) = r_2.
\]
A standard computation (cf.\,  \cite[Exercise 2.62(d)]{LP2016}) yields that 
\[
Y(\vec{e},\vec{f}) = Y(\vec{f},\vec{e}) = \frac{r_1 + r_2 - t}{2}.
\]
Substituting these expressions into the  determinant \eqref{eq:trans} yields the claimed formula
\[
\mathbb{P}_{\mathrm{UST}}\big[e,f\in \eta(\omega)\big]
=
\det
\begin{pmatrix}
r_1 & \frac{r_1+r_2-t}{2}\\[2pt]
\frac{r_1+r_2-t}{2} & r_2
\end{pmatrix}
=
r_1r_2 - \frac{1}{4}(r_1+r_2-t)^2,
\]
which completes the proof. \qedhere
\end{proof}

\subsubsection{Graph Laplacian and the Moore--Penrose inverse}\label{sec:La}

\begin{definition}\label{def:Laplacian}
Let $G=(V,E)$ be a finite graph with $|V|=n$.
The adjacency matrix $A_G$, degree matrix $D_G$ and Laplacian matrix $L_G$ are $n\times n$ matrices indexed by $V$ defined as follows:
\begin{itemize}
	\item $(A_G)_{u,v}$ is the number of edges between $u$ and $v$;
	\item $(D_G)_{u,u}=\deg(u;G)$ and $(D_G)_{u,v}=0$ for $u\neq v$; 
	\item $L_G=D_G-A_G$.
\end{itemize}
\end{definition}
Since $L_G$ is real symmetric, its eigenvalues are all real and will be denoted $\lambda_1\leq \lambda_2\leq\cdots\leq \lambda_n$. We have the following well-known properties of $L_G$, for instance see \cite[Theorem~2.1 and 2.2]{Mohar1991}. 
\begin{proposition}\label{prop:Laplacian}
Let $G$ be a finite graph and let $L_G$ be the associated Laplacian matrix. Then the following conclusions hold. 
\begin{enumerate}
	\item[(1)]  $L_G$ is positive semi-definite.
	
	\item[(2)] The smallest eigenvalue $\lambda_1=0$ and a corresponding eigenvector is $\mathbf{1}=(1,\ldots,1)^{\mathsf T}$. Here $M^{\mathsf T}$ denotes the transpose of $M$. 
	
	\item[(3)] The multiplicity of $0$ as an eigenvalue of $L_G$ is equal to the number of connected components of $G$. In particular, $\lambda_2>0$ if and only if $G$ is connected. 
	
	\item[(4)] The largest eigenvalue $\lambda_n\leq \max\big\{\deg(u;G)+\deg(v;G) \colon \{u,v\}\in E(G) \big\}$\,.
\end{enumerate}

\end{proposition}

Since $L_G$ is singular, we shall need the Moore--Penrose inverse \cite{Moore1920GeneralReciprocal,Penrose1955GeneralizedInverse}. 
\begin{definition}\label{def:MPinverse}
For any real $m\times n$ matrix $A$, the Moore--Penrose inverse  $A^+$ is the unique $n\times m$ matrix satisfying the four conditions:
\begin{enumerate}
	\item[(1)] $AA^+A=A$;
	\item[(2)] $A^+AA^+=A^+$;
	\item[(3)] $(AA^+)^{\mathsf T}=AA^+ $;
	\item[(4)] $(A^+A)^{\mathsf T}=A^+A$.
\end{enumerate}
When $A$ is square and invertible, $A^+=A^{-1}$. In the singular or rectangular case, $A^+$ generalizes the inverse. 
\end{definition}

\begin{lemma}\label{lem:Lalpacian-example}
Consider complete graphs $K_n$ with $n\ge5$. Let $S\subset E(K_n)$ and write $H\coloneq \big(V(K_n),S\big)$ and  $G\coloneq \big(V(K_n),E(K_n)\setminus S\big)$. Assume $S$ is such that $\max\{\deg(x;H)\mid x\in V(H)\}\leq 2$ and $G$ is still connected.  Let $\mathbf{1}^\perp$ be the orthonormal complement of the all one-vector $\operatorname{\mathbf{1}}$ in the space $l_2\big(V(K_n)\big)$. Then $I-\frac{1}{n}L_H$ is invertible and 
\[
L_{G}^+\mathbf{x}=\frac{1}{n}\big(I-\frac{1}{n}L_H\big)^{-1}\mathbf{x}\,,\,\,\,\forall\,\mathbf{x}\in \mathbf{1}^\perp\,.
\]
\end{lemma}
\begin{proof}
By item (4) in Proposition~\ref{prop:Laplacian}, the largest eigenvalue of the real symmetric matrix $L_H$ is at most $4$. Hence all the eigenvalues of the real symmetric matrix $I-\frac{1}{n}L_H$ are  at least $1-\frac{4}{n}\geq \frac{1}{5}>0$. Thus  $I-\frac{1}{n}L_H$ is invertible.

Note that $L_G+L_H=L_{K_n}$ and $L_{K_n}=nI-J$, where $I$ is the identity matrix and $J$ is the all-ones matrix. For any $\mathbf{x}\in \mathbf{1}^\perp$, we have $J\mathbf{x}=\mathbf{0}$ and 
\[
L_G\mathbf{x}=(nI-J)\mathbf{x}-L_H\mathbf{x}=(nI-L_H)\mathbf{x}\,.
\]
Since $G$ is connected, by item (3) in Proposition~\ref{prop:Laplacian} we have $L_G$ is positive definite on the subspace $\operatorname{span}(\mathbf{1})^\perp$. Thus  we have 
\[
L_{G}^+\mathbf{x}=(nI-L_H)^{-1}\mathbf{x}=\frac{1}{n}\big(I-\frac{1}{n}L_H\big)^{-1}\mathbf{x}\,,\,\,\,\forall\,\mathbf{x}\in \mathbf{1}^\perp\,.\qedhere
\]
\end{proof}
Next we recall a lemma which relates effective resistance with graph Laplacian; see \cite[Theorem~2.1]{ESMJK2011} or equation (7) in \cite{GX2004}.
\begin{lemma}\label{lem:res-Lap}
Let $u,v$ be two distinct vertices in a finite graph $G=(V,E)$. Write $\mathbf{e}_u$ for the vector in $l_2(V)$ that has entry $1$  at position $u$  and zeroes elsewhere.  Then 
\[
\reff(u\leftrightarrow v; G)=(\mathbf{e}_u-\mathbf{e}_v)^TL_G^+\,(\mathbf{e}_u-\mathbf{e}_v)\,.
\]
\end{lemma}

\subsubsection{Edmonds' matroid polytope theorem}\label{sec:matroid}
Matroid theory turns out to be a powerful tool when studying correlations \cite{HSW2022}. Our proof of Theorem~\ref{thm:UST-sharp} provides another example of applications of matroid theory in probability. 
\begin{definition}\label{def:matroid}
A matroid $\sM=(E,\cI)$ consists of a finite ground set $E$ and a collection $\cI\subset 2^{E}$ of independent sets satisfying:
\begin{enumerate}
	\item[(1)] $\emptyset\in\cI$;
	
	\item[(2)] if $A\in\cI$ and $B\subset A$, then $B\in\cI$;
	
	\item[(3)] if $A,B\in\cI$ and $|A|<|B|$, then $\exists e\in B\setminus A$ such that $A\cup\{e\}\in\cI$. 
\end{enumerate}
The rank function $r:2^E\to\bbZ_{\ge0}$ associated with $\sM$ is given by 
\[
r(S)\coloneq \max\left\{ |A|  \,\,\big|\,\, A\subset S,A\in\cI  \right\}\,.
\]
\end{definition}

A graphical matroid $\sM(G)$ of a  graph $G=(V,E)$ has ground set $E$; a set of edges is independent if and only if it contains no cycle (i.e., it is a forest). Its rank function 
\[
r(S)=|V|-c(S)\,,
\]
where $c(S)$  is the number of connected components in the subgraph $(V,S)$. 
Now we recall Edmonds' matroid polytope theorem \cite{Edmonds1970}.
\begin{theorem}\label{thm:Edmonds}
Let $\sM=(E,\cI)$ be a matroid and $r$ be the rank function. Then
the convex hull of characteristic vectors of independent sets of a matroid $\sM$ is the polytope
\[
\left\{ x\in\bbR^E \, \Big|\, x\ge0, \sum_{e\in S}x_e\leq r(S) \mathrm{ ~for~all~ }S\subset E  \right\}\,.
\]
In particular, if $\sM=\sM(G)$ is a graphical matroid associated with a graph $G=(V,E)$, then 
\be\label{eq:Edmonds}
\operatorname{conv}\{\mathbf 1_F: F\subset E \text{ is a forest}\}
=
\left\{x\in \mathbb R_{\ge 0}^E: \sum_{e\in S}x_e\le r(S)\text{ for all }S\subset E \right \}.
\ee
\end{theorem}

\subsection{Proof of Theorem~\ref{thm:UST}}\label{sec:pf-UST}
\begin{proof}
Since $X$ is chosen uniformly from $V$ and independently of the UST sample $\omega$, it suffices to prove that there exists a constant $C>0$ such that
\be\label{eq:thm1goal}
\sum_{x\in V}\EUST\big[\deg(x;\omega)^2\big]\leq C|V|\,.
\ee
Because $\PUST$ has the p-NC property, Corollary~\ref{cor:UST} gives
\[
\sum_{x\in V}\EUST\big[\deg(x;\omega)^2\big]\leq 2(|V|-1)+\sum_{\langle e,f\rangle \colon e\sim f}\PUST\big[\omega(e)=1\big]\cdot\PUST\big[\omega(f)=1\big]\,.
\]
Using Kirchhoff's effective resistance formula \eqref{eq:K-formula} and the Cauchy--Schwarz inequality, we obtain
\[
\sum_{\langle e,f\rangle \colon e\sim f}\PUST\big[\omega(e)=1\big]\cdot\PUST\big[\omega(f)=1\big]
\leq \sum_{\langle e,f\rangle \colon e\sim f}\frac{\reff(e)^2+\reff(f)^2}{2}\,.
\]
Since $G$ is simple and $d$-regular,
\[
\sum_{\langle e,f\rangle \colon e\sim f}\frac{\reff(e)^2+\reff(f)^2}{2}=2(d-1)\sum_{e\in E}\reff(e)^2\,,
\]
where the sum $\sum_{\langle e,f\rangle \colon e\sim f}$ is over ordered pairs of adjacent edges in $G$.

Now partition $E$ into the sets $E_i$ from Corollary~\ref{cor:E_i}. Then
\begin{align*}
\sum_{e\in E}\reff(e)^2&=\sum_{i=0}^{M}\sum_{e\in E_i}\reff(e)^2\leq \sum_{i=0}^{M}|E_i|\cdot \epsilon_{i+1}^2\\
&\leq |E|\epsilon_1^2+\sum_{i=1}^{M}\frac{|V|}{3^{i-1}}\epsilon_{i+1}^2\\
&\leq\frac{d|V|}{2}\cdot \frac{3^2}{d^2}+|V|\sum_{i=1}^{M}\frac{1}{3^{i-1}}\cdot \frac{3^{2i+2}}{d^2}\leq \frac{45|V|}{d}\,.
\end{align*}
Putting everything together,
\begin{align*}
\sum_{x\in V}\EUST\big[\deg(x;\omega)^2\big]
&\leq 2(|V|-1)+90|V|\le 92|V|\,.
\end{align*}
Thus \eqref{eq:thm1goal} holds with $C=92$, completing the proof. 
\end{proof}

\subsection{Proof of Theorem~\ref{thm:UST-sharp}}
\label{sec:pf-USTsharp}
In this section, we establish that the sharp constant for the upper bound of the second moment is exactly $6$. The proof has two parts: first proving the upper bound $6$, then constructing a sequence of regular graphs where the second moment tends to $6$, showing sharpness.

\begin{proposition}\label{prop:appendix-UST-sharp}
Under the assumptions of Theorem~\ref{thm:UST},
\be\label{eq:thmUSTsharp}
\bbE\big[\deg(X;\omega)^2\big] < 6,
\ee
\end{proposition}

\begin{proposition}\label{prop:appendix-tightness}
There exists a sequence $G_d$ of finite simple connected $d$-regular graphs with
$d\to\infty$ such that, if $\omega_d\sim\UST(G_d)$ and $X_d$ is uniformly chosen from
$V(G_d)$ independently of $\omega_d$, then
\[
\lim_{d\to\infty}\bbE\big[\deg(X_d;\omega_d)^2\big]=6.
\]
\end{proposition}

\begin{proof}[Proof of Theorem~\ref{thm:UST-sharp}]
This follows immediately from Proposition~\ref{prop:appendix-UST-sharp} and Proposition~\ref{prop:appendix-tightness}.
\end{proof}

\subsubsection{Proof of the second moment bound $6$} \label{sec:UST_upper}

Assume that $G$  is a  finite, connected, regular graph with degree $d$. 
If $d\in\{1,2\}$, then   \eqref{eq:thmUSTsharp}
is trivial since $\deg(X;\omega)^2\leq d^2\leq 4$. Hence we assume throughout that $d\ge 3$.

\begin{lemma}\label{lem:appendix-second-moment}
Assume  $(G,\omega,X)$ satisfies the assumptions of Theorem~\ref{thm:UST}, and conditional on $X$ let $Y, Z$ be two distinct neighbors of $X$ chosen uniformly among all ordered pairs of distinct neighbors (independent of $\omega$). Write $n=|V|$ for simplicity. Define
\[
p_1 := \mathbb{P}\big[ \{X,Y\} \in \eta(\omega) \text{ and } \{X,Z\} \in \eta(\omega) \big].
\]
Then
\[
\mathbb{E}\big[\deg(X;\omega)^2\big] = 2 - \frac{2}{n} + d(d-1)p_1.
\]
\end{lemma}

\begin{proof}
We start from the identity valid for the UST on any finite, connected graph (see \eqref{eq:relation-1-ad}):
\be\label{eq:A1}
\sum_{v\in V} \mathbb{E}_{\mathrm{UST}}\big[\deg(v;\omega)^2\big] = 2(n-1) + \sum_{\langle e,f\rangle \colon e\sim f} \mathbb{P}_{\mathrm{UST}}\big[e,f\in \eta(\omega)\big],
\ee
where the sum $\sum_{\langle e,f\rangle \colon e\sim f}$ runs over all ordered pairs of distinct adjacent edges.

Fix a vertex $v$. There are exactly $d(d-1)$ ordered pairs of distinct edges incident to $v$. Then by definition of $p_1$ (with $X$ uniform and $Y,Z$ uniform ordered neighbors given $X$),
\[
p_1 = \frac{1}{n}\sum_{v\in V} \frac{1}{d(d-1)} \sum_{\substack{\text{ordered }(e,f)\\ \text{incident to }v}} \mathbb{P}_{\mathrm{UST}}[e,f\in \eta(\omega)].
\]
Multiplying both sides by $nd(d-1)$ yields
\[
n d(d-1) p_1 = \sum_{v\in V} \sum_{\substack{\text{ordered }(e,f)\\ \text{incident to }v}} \mathbb{P}_{\mathrm{UST}}[e,f\in \eta(\omega)] = \sum_{\langle e,f\rangle \colon e\sim f} \mathbb{P}_{\mathrm{UST}}\big[e,f\in \eta(\omega)\big].
\]
Substituting this into the identity \eqref{eq:A1} yields
\[
\sum_{v\in V} \mathbb{E}_{\mathrm{UST}}\big[\deg(v;\omega)^2\big] = 2(n-1) + n d(d-1) p_1.
\]
Averaging over the uniformly chosen vertex $X$ (independent of $\omega$) yields the claimed formula
\[
\mathbb{E}\big[\deg(X;\omega)^2\big] = \frac{1}{n}\sum_{v\in V} \mathbb{E}_{\mathrm{UST}}\big[\deg(v;\omega)^2\big] = 2 - \frac{2}{n} + d(d-1)p_1. 	\qedhere
\]
\end{proof}

We next record the law of the ordered wedge
$(Y,X,Z)$ via a simple random-walk interpretation.
\begin{lemma}\label{lem:appendix-wedge-law}
Let $(X_0,X_1,X_2)$ be a two-step simple random walk on $G$, started from the uniform
distribution on $V$. Then the ordered triple $(Y,X,Z)$ in Lemma~\ref{lem:appendix-second-moment} has the same distribution as $(X_0,X_1,X_2)$ conditioned on the event $\{X_2\neq X_0\}$.
Moreover,
\[
\bbP[X_2\neq X_0]=1-\frac{1}{d}.
\]
\end{lemma}

\begin{proof}Write $n=|V|$ for simplicity.
Fix a triple $(y,x,z)$ such that $y\sim x$, $z\sim x$, and $y\neq z$. By definition,
\[
\bbP[Y=y,X=x,Z=z]
=
\frac{1}{n}\cdot \frac{1}{d(d-1)}.
\]
On the other hand,
\[
\bbP[X_0=y,X_1=x,X_2=z]
=
\frac{1}{n}\cdot \frac{1}{d}\cdot \frac{1}{d}
=
\frac{1}{nd^2}.
\]
Since $G$ is simple and $d$-regular, conditioned on $X_1=x$, the probability that the
second step returns to $X_0$ equals $1/d$. Hence
\[
\bbP[X_2\neq X_0]=1-\frac{1}{d},
\]
and therefore
\[
\bbP\big[(X_0,X_1,X_2)=(y,x,z)\mid X_2\neq X_0\big]
=
\frac{1/(nd^2)}{1-1/d}
=
\frac{1}{nd(d-1)}.
\]
This matches the law of $(Y,X,Z)$.
\end{proof}

For the random ordered wedge $(Y,X,Z)$ in Lemma~\ref{lem:appendix-second-moment}, define
\[
r_1\coloneq \reff(X\leftrightarrow Y;G),\qquad
r_2\coloneq \reff(X\leftrightarrow Z;G),\qquad
t\coloneq \reff(Y\leftrightarrow Z;G).
\]
Set
\[
a\coloneq \frac{2}{d+1},
\qquad
r_1=a+\alpha,
\quad
r_2=a+\beta,
\quad
t=a+\gamma.
\]
Nachmias and Peres noted a useful lower bound on the effective resistance between the two endpoints of an edge (see \cite[Eq.~(6)]{NP2022}): for any two distinct vertices $u,v$ in a finite graph,
\[
\reff(u\leftrightarrow v;G)
\ge
\frac{1}{\deg(u)+1}+\frac{1}{\deg(v)+1}.
\]
Since $G$ is $d$-regular, this implies $\alpha,\beta,\gamma\ge0$.

\begin{lemma}\label{lem:appendix-alpha-gamma}
Write $n=|V|$ for simplicity. With the notation above, 
we have
\[
\bbE[\alpha]
=
\frac{2(n-d-1)}{nd(d+1)}
\]
and
\[
\bbE[\gamma]
=
\frac{4(n-d-1)}{n(d-1)(d+1)}.
\]
\end{lemma}

\begin{proof}
The pair $(X,Y)$ is a uniformly chosen oriented edge. Hence, by Foster's theorem \eqref{eq:Foster},
\[
\bbE[r_1]
=
\frac{1}{nd}\sum_{x\in V}\sum_{y\sim x}\reff(x\leftrightarrow y;G)
=
\frac{2}{nd}\sum_{e\in E}\reff(e)
=
\frac{2(n-1)}{nd}.
\]
Subtracting $a=2/(d+1)$ gives the formula for $\bbE[\alpha]$.

It remains to compute $\bbE[\gamma]$. Let
\[
K\coloneq \frac{r_1+r_2-t}{2}.
\]
This is the transfer-current kernel between the two oriented edges $(X,Y)$ and $(X,Z)$.
Fix an oriented edge $\vec{e}=(x,y)$ and write $\vec{f}_z=(x,z)$ for an oriented edge from $x$ to a
neighbor $z$. Kirchhoff's node law for the unit current sent from $x$ to $y$
gives
\[
\sum_{z\sim x} Y(\vec{e},\vec{f}_z)=1,
\]
Since $Y(\vec{e},\vec{e})=\reff(x\leftrightarrow y;G)$, we have 
\[
\sum_{\substack{z\sim x\\ z\ne y}} Y(\vec{e},\vec{f}_z)
=
1-\reff(x\leftrightarrow y;G).
\]
Averaging over the uniformly chosen oriented edge $(X,Y)$ and then over the uniformly
chosen $Z\ne Y$, we get
\[
\bbE[K]
=
\frac{1}{nd(d-1)}
\sum_{x\in V}\sum_{y\sim x}\bigl(1-\reff(x\leftrightarrow y;G)\bigr)
=
\frac{nd-2(n-1)}{nd(d-1)}.
\]
Since
\[
2K=a+\alpha+\beta-\gamma
\]
and $\alpha,\beta$ have the same distribution,
\[
\bbE[\gamma]=a+2\bbE[\alpha]-2\bbE[K].
\]
Substitution gives the displayed formula for $\bbE[\gamma]$.
\end{proof}

\begin{lemma}\label{lem:appendix-alpha-beta}
We have
\[
\bbE[\alpha\beta]
\le
\frac{(d+2)(n-d-1)}{n d^2(d+1)}.
\]
\end{lemma}

\begin{proof}
For every edge $e=\{u,v\}$, set
\[
\alpha_e\coloneq \reff(u\leftrightarrow v;G)-a.
\]
When an edge is specified by its endpoints, we use the shorthand
$\alpha_{uv}:=\alpha_{\{u,v\}}$; in particular, $\alpha_{uv}=\alpha_{vu}$.
As above, $\alpha_e\ge0$. Let $T\sim\UST(G)$, and let $\mathbf 1_T\in\{0,1\}^E$ denote
its edge indicator vector. Kirchhoff's formula \eqref{eq:K-formula} becomes
\[
\EUST[(\mathbf 1_T)_e]=\reff(e).
\]
Thus $r=(\reff(e))_{e\in E}$ is the expectation of spanning-tree incidence vectors. In
particular, for every $S\subseteq E$,
\[
\sum_{e\in S}\reff(e)
=
\EUST[|T\cap S|]
\le
\operatorname{rank}(S),
\]
where $\operatorname{rank}(S)$ is the rank of $S$ in the graphic matroid $\sM(G)$. Hence
\be\label{eq:alphae}
\sum_{e\in S}\alpha_e
\le
\sum_{e\in S}\reff(e)
\le
\operatorname{rank}(S)
\qquad(\forall\,S\subseteq E).
\ee
Recall that by Edmonds' matroid polytope theorem, we have \eqref{eq:Edmonds}:
\[
\operatorname{conv}\{\mathbf 1_F: F\subset E \text{ is a forest}\}
=
\left\{x\in \mathbb R_{\ge 0}^E:\sum_{e\in S}x_e\le \operatorname{rank}(S)\text{ for all }S\subset  E \right \}.
\]
Write $\boldsymbol{\alpha}=(\alpha_e)_{e\in E}$. 
Since $\alpha_e\ge0$ for every $e\in E$,  \eqref{eq:alphae} implies that 
\[
\boldsymbol{\alpha}\in 	\left\{x\in \mathbb R_{\ge 0}^E:\sum_{e\in S}x_e\le \operatorname{rank}(S)\text{ for all }S\subset  E \right \}\,.
\]
By Edmonds' matroid polytope theorem we have 
\[
\boldsymbol{\alpha}\in
\operatorname{conv}\{\mathbf 1_F: F\subset E \text{ is a forest}\}.
\]
That is,  there are coefficients $\lambda_F\ge 0$, indexed by forests
$F\subseteq E$ and summing to one, such that
\[
\boldsymbol{\alpha}=\sum_F \lambda_F \mathbf 1_F .
\]
Choosing $F$ with probability $\lambda_F$ gives a random forest
$\mathcal F$ satisfying
\[
\mathbb P[e\in \mathcal F]
=
\sum_{F\ni e}\lambda_F
=
\alpha_e .
\]

For $x\in V$, define
\[
A_x\coloneq\sum_{y\sim x}\alpha_{xy},
\qquad
D_x\coloneq\deg(x;\mathcal F).
\]
Then $A_x=\bbE[D_x]$, and Jensen's inequality gives
\[
\sum_{x\in V}A_x^2
\le
\bbE\left[\sum_{x\in V}D_x^2\right].
\]
We next use a deterministic estimate for forests. If $F\subseteq E$ is a forest,
$m=|F|$, and $D_x=\deg(x;F)$, then $\sum_xD_x=2m$. Let $c$ denote the number of connected components of $F$ that contain at least one edge. If these components have $(m_1,\ldots,m_c)$ edges, respectively, then each has $m_i+1$ vertices. Consequently,
\[
\sum_{x:D_x>0}(D_x-1)=\sum_{i=1}^{c}{\big[2m_i-(m_i+1)\big]}=m-c.
\] 
For $D_x>0$, the inequality $D_x\le d$ gives
\[
D_x(D_x-1)\le d(D_x-1).
\]
Therefore
\[
\sum_{x\in V}D_x^2
=
2m+\sum_{x:D_x>0}D_x(D_x-1)
\le
2m+d(m-c)
\le
(d+2)m.
\]
Applying this to $\mathcal F$ and taking expectations,
\[
\sum_{x\in V}A_x^2
\le
(d+2)\bbE[|\mathcal F|]
=
(d+2)\sum_{e\in E}\alpha_e.
\]
By Foster's theorem \eqref{eq:Foster} and $|E|=nd/2$,
\[
\sum_{e\in E}\alpha_e
=
(n-1)-\frac{2}{d+1}\cdot\frac{nd}{2}
=
\frac{n-d-1}{d+1}.
\]
Thus
\[
\sum_{x\in V}A_x^2
\le
\frac{(d+2)(n-d-1)}{d+1}.
\]

Finally, condition on $X=x$. Since $(Y,Z)$ is uniformly distributed over ordered distinct
neighbors of $x$,
\[
\bbE[\alpha\beta\mid X=x]
=
\frac{1}{d(d-1)}
\left(A_x^2-\sum_{y\sim x}\alpha_{xy}^2\right).
\]
By Cauchy--Schwarz,
\[
\sum_{y\sim x}\alpha_{xy}^2\ge \frac{A_x^2}{d}.
\]
Consequently,
\[
\bbE[\alpha\beta]
\le
\frac{1}{nd^2}\sum_{x\in V}A_x^2
\le
\frac{(d+2)(n-d-1)}{n d^2(d+1)}. \qedhere
\]
\end{proof}

\begin{proof}[Proof of Proposition~\ref{prop:appendix-UST-sharp}]
By Lemma~\ref{lem:appendix-transfer-current},
\[
p_1
=
\bbE\left[r_1r_2-\frac{1}{4}(r_1+r_2-t)^2\right].
\]
After substituting $r_1=a+\alpha$, $r_2=a+\beta$, and $t=a+\gamma$, this becomes
\begin{align*}
p_1
&=
\frac{3a^2}{4}
+a\bbE[\alpha]
+\frac{a}{2}\bbE[\gamma]
+\bbE[\alpha\beta]
-\frac{1}{4}\bbE[(\alpha+\beta-\gamma)^2] \\
&\le
\frac{3a^2}{4}
+a\bbE[\alpha]
+\frac{a}{2}\bbE[\gamma]
+\bbE[\alpha\beta].
\end{align*}
Using Lemmas~\ref{lem:appendix-alpha-gamma} and~\ref{lem:appendix-alpha-beta}, we obtain
\begin{align*}
d(d-1)p_1
\le\;&
\frac{3d(d-1)}{(d+1)^2}
+
\frac{4(d-1)(n-d-1)}{n(d+1)^2} \\
&+
\frac{4d(n-d-1)}{n(d+1)^2}
+
\frac{(d-1)(d+2)(n-d-1)}{n d(d+1)}.
\end{align*}
Combining this with Lemma~\ref{lem:appendix-second-moment} yields
\begin{align*}
\bbE\big[\deg(X;\omega)^2\big]
\le\;&
2-\frac{2}n
+
\frac{3d(d-1)}{(d+1)^2} \\
&+
\frac{4(2d-1)(n-d-1)}{n(d+1)^2}
+
\frac{(d-1)(d+2)(n-d-1)}{n d(d+1)}.
\end{align*}
Since $G$ is simple and $d$-regular, $n\ge d+1$, so
\[
0\le \frac{n-d-1}{n}\le1.
\]
The right-hand side is increasing in $(n-d-1)/n$, and $-2/n\le0$. Hence
\begin{align*}
\bbE\big[\deg(X;\omega)^2\big]
&\le
2+
\frac{3d(d-1)+4(2d-1)}{(d+1)^2}
+
\frac{(d-1)(d+2)}{d(d+1)} \\
&=
6-\frac{d+7}{(d+1)^2}-\frac{2}{d(d+1)}
<6.
\end{align*}
This proves the proposition.
\end{proof}

\subsubsection{Sharpness of the constant $6$}
\label{sec:UST_lower}

We now prove Proposition~\ref{prop:appendix-tightness}, i.e.,  the constant $6$ in Proposition~\ref{prop:appendix-UST-sharp} is sharp.

For each odd integer $d$, put $q=d+2$. We first construct a one-port dense block
$B_d$ on $q$ vertices. Start from the complete graph $K_q$, choose three distinct vertices $o,a,b$,
delete the two edges $\{o,a\}$ and $\{o,b\}$, and delete a perfect matching on
$V(K_q)\setminus\{o,a,b\}$, say $ \big\{  \{u_i,v_i\},i=1,\ldots,\frac{q-3}{2} \big\}$. The parity assumption on $d$ is used only here, to make
this perfect matching available. Write $S=\big\{\{o,a\},\{o,b\},\{u_i,v_i\},i=1,\ldots,\frac{q-3}{2}\big\}$ for the set of deleted edges. Let $H_d\coloneq \big(V(K_q),S\big)$ be the deleted subgraph and 
set
\[
B_d\coloneq \big(V(K_q),E(K_q)\setminus S\big)\,.
\]
The vertex $o$ is the port. By construction,
\[
\deg(o;B_d)=d-1,
\qquad
\deg(v;B_d)=d\;\;\big(\,\forall\,v\in V(K_q)\setminus\{o\}\big).
\]
\begin{example}\label{example}
Take $d$ disjoint copies $B_d^{(1)},\ldots,B_d^{(d)}$ with ports
$o_1,\ldots,o_d$, add one new vertex $c$, and add the $d$ edges
\[
\{c,o_1\},\ldots,\{c,o_d\}.
\]
Call the resulting graph $G_d$. Since the non-port vertices already have degree $d$,
each port has degree $d-1$ inside its block and receives one bridge, and $c$ has degree
$d$, the graph $G_d$ is finite, simple, connected, and $d$-regular. Also
\[
|V(G_d)|=1+dq=1+d(d+2).
\]
\end{example}
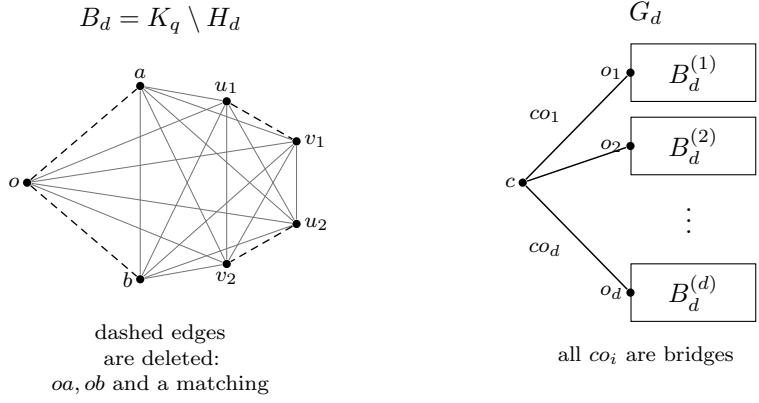
\begin{figure}[htbp]
\centering
\begin{tikzpicture}[
scale=0.88,
every node/.style={font=\small},
every label/.style={font=\scriptsize,inner sep=1pt},
vertex/.style={circle,fill=black,inner sep=0pt,minimum size=3.1pt},
port/.style={circle,fill=black,inner sep=0pt,minimum size=3.1pt},
block/.style={draw=black,minimum width=1.65cm,minimum height=0.76cm},
present/.style={draw=black!55,line width=0.3pt},
bridge/.style={draw=black,line width=0.55pt},
deleted/.style={draw=black,densely dashed,line width=0.5pt}
]
\begin{scope}
\node at (0.25,2.45) {$B_d=K_q\setminus H_d$};

\node[vertex,label=left:$o$] (o) at (-1.75,0) {};
\node[vertex,label=above:$a$] (a) at (-0.05,1.45) {};
\node[vertex,label=left:$b$] (b) at (-0.05,-1.45) {};
\node[vertex,label=above:$u_1$] (u1) at (1.25,1.22) {};
\node[vertex,label=right:$v_1$] (v1) at (2.30,0.62) {};
\node[vertex,label=right:$u_2$] (u2) at (2.30,-0.62) {};
\node[vertex,label=below:$v_2$] (v2) at (1.25,-1.22) {};

\foreach \x/\y in {a/b,a/u1,a/v1,a/u2,a/v2,b/u1,b/v1,b/u2,b/v2,
	u1/u2,u1/v2,v1/u2,v1/v2,o/u1,o/v1,o/u2,o/v2}
\draw[present] (\x) -- (\y);

\draw[deleted] (o) -- (a);
\draw[deleted] (o) -- (b);
\draw[deleted] (u1) -- (v1);
\draw[deleted] (u2) -- (v2);

\node[align=center,font=\scriptsize,text width=3.25cm] at (0.25,-2.65)
{dashed edges are deleted:\\ $oa,ob$ and a matching};
\end{scope}

\begin{scope}[xshift=5.70cm]
\node at (1.85,2.55) {$G_d$};
\node[vertex,label=left:$c$] (c) at (0,0) {};

\node[block,anchor=west] (B1) at (1.62,1.65) {$B_d^{(1)}$};
\node[block,anchor=west] (B2) at (1.62,0.55) {$B_d^{(2)}$};
\node at (2.50,-0.45) {$\vdots$};
\node[block,anchor=west] (Bd) at (1.62,-1.65) {$B_d^{(d)}$};

\node[port,label=left:$o_1$] (o1) at (1.62,1.65) {};
\node[port,label=left:$o_2$] (o2) at (1.62,0.55) {};
\node[port,label=left:$o_d$] (od) at (1.62,-1.65) {};

\draw[bridge] (c) -- (o1) node[pos=0.45,above left,font=\scriptsize] {$co_1$};
\draw[bridge] (c) -- (o2);
\draw[bridge] (c) -- (od) node[pos=0.45,below left,font=\scriptsize] {$co_d$};
\node[font=\scriptsize] at (1.85,-2.60)
{all $co_i$ are bridges};
\end{scope}
\end{tikzpicture}
\caption{The sharpness construction for the constant $6$. The left panel shows the one-port dense block $B_d$: start with $K_q$, delete the two edges $\{o,a\},\{o,b\}$, and delete a perfect matching on the remaining $q-3=d-1$ vertices. The right panel shows $G_d$, obtained by attaching $d$ disjoint copies of $B_d$ to a new central vertex $c$ through their ports.}
\label{fig:appendix-tightness-construction}
\end{figure}
\begin{lemma}\label{lem:appendix-dense-block}
For the blocks $B_d$ constructed above, $B_d$ is connected for all sufficiently large
odd $d$. If $T_d\sim\UST(B_d)$, then
\[
\frac{1}q\sum_{v\in V(B_d)}\bbE\big[\deg(v;T_d)^2\big]
=
5+O(q^{-1}).
\]
\end{lemma}

\begin{proof}
Assume $q=d+2\ge5$. We write $L_{B_d}$ and $L_{H_d}$ for the graph Laplacians, viewed
as self-adjoint operators on $l_2{\big(V(B_d)\big)}$. 
Since $L_{H_d}$ is real symmetric with eigenvalues $0=\lambda_1\leq \cdots\leq \lambda_{\max}$, we have $\|L_{H_d}\|_{2\to 2}=\lambda_{\max}$. 
Since the deleted graph
$H_d$ consists of the two edges incident to $o$ and a matching on the remaining
vertices, its maximum degree $\Delta(H_d)$ is $2$. Hence by item (4) in Proposition~\ref{prop:Laplacian}, we have
\[
\|L_{H_d}\|_{2\to 2}\le 2\Delta(H_d)\le 4.
\]
By Lemma~\ref{lem:Lalpacian-example}, $I-q^{-1}L_{H_d}$ is invertible. Set $A_d=q^{-1}L_{H_d}$.
Then $\|A_d\|_{2\to2}\le 4/q$. 
For the partial sums
$S_N\coloneq \sum_{k=0}^N A_d^k$,
\[
(I-A_d)S_N=S_N(I-A_d)=I-A_d^{N+1}.
\]
Since $\|A_d^{N+1}\|_{2\to2}\le \|A_d\|_{2\to2}^{N+1}\to0$ for $q>4$,
we get
\[
(I-A_d)^{-1}
=
\sum_{k\ge0}A_d^k.
\]
Moreover, for all sufficiently large $q$, say $q\ge8$,
\[
\|(I-A_d)^{-1}-I\|_{2\to2}
\le
\sum_{k\ge1}\|A_d\|_{2\to2}^k
\le
\frac{4/q}{1-4/q}
\le
\frac{8}q.
\]
Hence if we define $R_q$ via 
\[
\frac{1}q\left(I-\frac{1}qL_{H_d}\right)^{-1}
=
\frac{1}qI+R_q,
\]
then 
\[
\|R_q\|_{2\to2}=O(q^{-2}).
\]

For any distinct vertices $u,v$ the vector $\mathbf e_u-\mathbf e_v$ lies in
$\mathbf 1^\perp$ and has $l_2$ norm $\sqrt{2}$. 
By Lemma~\ref{lem:Lalpacian-example}, 
\[
L_{B_d}^+ (\mathbf e_u-\mathbf e_v)
=
\frac{1}q\left(I-\frac{1}qL_{H_d}\right)^{-1}(\mathbf e_u-\mathbf e_v)
=
\frac{1}q (\mathbf e_u-\mathbf e_v)+R_q(\mathbf e_u-\mathbf e_v).
\]
By Lemma~\ref{lem:res-Lap}, we have  
\[
\reff(u\leftrightarrow v;B_d)
=	(\mathbf e_u-\mathbf e_v)^{\mathsf T}L_{B_d}^+(\mathbf e_u-\mathbf e_v)\,.
\]
Hence  uniformly over distinct
$u,v$,
\begin{align}\label{eq:res-est}
\reff(u\leftrightarrow v;B_d)
&=	(\mathbf e_u-\mathbf e_v)^{\mathsf T}\left[ \frac{1}q (\mathbf e_u-\mathbf e_v)+R_q(\mathbf e_u-\mathbf e_v)\right]\nonumber\\
&= \frac{2}{q}+(\mathbf e_u-\mathbf e_v)^{\mathsf T}R_q(\mathbf e_u-\mathbf e_v)\nonumber\\
&=\frac{2}{q}+O(q^{-2})\,.
\end{align}
Kirchhoff's formula \eqref{eq:K-formula} gives
\be\label{eq:deg-o-Td}
\EUST[\deg(o;T_d)]
=
\sum_{y\sim o}\reff(o\leftrightarrow y;B_d)
=
(q-3)\left(\frac{2}q+O(q^{-2})\right)
=
2+O(q^{-1}).
\ee

It remains to estimate the average second moment. For two distinct edges
$e=\{x,y\}$ and $f=\{x,z\}$ incident to $x$, orient both edges away from $x$.
Lemma~\ref{lem:appendix-transfer-current} gives
\[
\PUST[e,f\in T_d]
=
r_1r_2-\frac{1}{4}(r_1+r_2-t)^2,
\]
where
\[
r_1=\reff(x\leftrightarrow y;B_d),\quad
r_2=\reff(x\leftrightarrow z;B_d),\quad
t=\reff(y\leftrightarrow z;B_d).
\]
The resistance estimate \eqref{eq:res-est} above gives
\[
r_1=\frac{2}q+O(q^{-2}),\qquad
r_2=\frac{2}q+O(q^{-2}),\qquad
t=\frac{2}q+O(q^{-2}),
\]
and therefore
\[
\PUST[e,f\in T_d]
=
\left(\frac{2}q+O(q^{-2})\right)^2
-\frac{1}{4}\left(\frac{2}q+O(q^{-2})\right)^2
=
\frac{3}{q^2}+O(q^{-3})
\]
uniformly over ordered incident pairs.
For $x\in V(B_d)$, the number of such ordered pairs is
$\deg(x;B_d)(\deg(x;B_d)-1)$.
Since the degrees in $B_d$ are $q-3$ at $o$
and $q-2$ elsewhere,
\[
\sum_{x\in V(B_d)}\deg(x;B_d)\big(\deg(x;B_d)-1\big)
=
(q-3)(q-4)+(q-1)(q-2)(q-3)
=q^3+O(q^2).
\]
Finally, every spanning tree on $q$ vertices has total degree $2(q-1)$, and
\[
\sum_{v\in V(B_d)}\deg(v;T_d)^2
=
\sum_{v\in V(B_d)}\deg(v;T_d)
+
\sum_{v\in V(B_d)}\deg(v;T_d)(\deg(v;T_d)-1).
\]
Taking expectations and using the preceding ordered-pair estimate yields
\[
\frac{1}q\sum_{v\in V(B_d)}\bbE[\deg(v;T_d)^2]
=
2-\frac{2}q
+
\frac{1}q\big(q^3+O(q^2)\big)\left(\frac{3}{q^2}+O(q^{-3})\right)
=
5+O(q^{-1}).\qedhere
\]
\end{proof}

\begin{proof}[Proof of Proposition~\ref{prop:appendix-tightness}]
It suffices to take odd $d$ and let $d\to\infty$. We will prove the result for the sequence of graphs $G_d$ in Example~\ref{example}. 

The edges $\{c,o_i\}$ are bridges, so every spanning tree of $G_d$ contains all of
them. After removing these bridges, what remains inside each
copy $B_d^{(i)}$ must  be a spanning tree of that copy. Conversely, choosing one
spanning tree in each block and then adding all bridges produces a spanning
tree of $G_d$. Thus the uniform measure on spanning trees of $G_d$ factors as
the product of the UST measures on the blocks, together with the deterministic
bridge set. In particular, $\deg(c;\omega_d)=d$ almost surely.

Let $T_d\sim\UST(B_d)$. By Lemma~\ref{lem:appendix-dense-block},
\[
\sum_{v\in V(B_d)}\bbE[\deg(v;T_d)^2]
=
q\big(5+O(q^{-1})\big)
=
5q+O(1),
\]
and
\[
\bbE[\deg(o;T_d)]\stackrel{\eqref{eq:deg-o-Td}}{=}2+O(q^{-1}).
\]
In $G_d$, each port receives one additional compulsory bridge edge. Hence the contribution
of one block to the total sum of squared degrees in the UST of $G_d$ is
\begin{align*}
&\sum_{v\in V(B_d)}\bbE[\deg(v;T_d)^2]
+
\bbE\!\left[(\deg(o;T_d)+1)^2-\deg(o;T_d)^2\right] \\
&\qquad =
5q+O(1)+2\bbE[\deg(o;T_d)]+1
=
5q+O(1).
\end{align*}
Therefore
\[
\bbE\big[\deg(X_d;\omega_d)^2\big]
=
\frac{d^2+d\big(5q+O(1)\big)}{1+dq}.
\]
Here the term $d^2$ is the contribution of the central vertex $c$, and the term
$d(5q+O(1))$ is the contribution of the $d$ identical blocks. Since $q=d+2$,
the last expression equals
\[
5+\frac{d}{q}+O(q^{-1})\longrightarrow 6.
\]
The universal upper bound in Proposition~\ref{prop:appendix-UST-sharp} gives
the matching upper bound for the supremum. Thus the supremum is exactly $6$,
although it is not attained by any finite simple connected regular graph covered
by Proposition~\ref{prop:appendix-UST-sharp}.
\end{proof}

\section*{Acknowledgments}
We thank Luca Makowiec for pointing out Aldous's limiting degree distribution and the resulting direct computation of the second moment. We also thank Anish Gupta for bringing his subsequent work to our attention.


\bibliographystyle{plain}
\bibliography{ncmst_ref}

\end{document}